\documentclass[reqno]{amsart}

\usepackage[a4paper,hmarginratio={1:1},vmarginratio={1:1}]{geometry}

\usepackage[utf8]{inputenc}
\usepackage[T1]{fontenc}
\usepackage{textcomp}
\usepackage{amsmath,amssymb,amsthm}
\usepackage{mathtools}
\usepackage{lmodern}
\usepackage{xcolor, pict2e}
\usepackage{microtype}
\usepackage{hyperref}
\hypersetup{pdfstartview=XYZ}
\usepackage{enumitem}
\usepackage{comment}

\usepackage{hyperref}
\hypersetup{
    colorlinks=true,
    linkcolor=blue,
    citecolor=magenta,
    urlcolor=blue,
    pdfborder={0 0 0}
}

\usepackage[font=sf, labelfont={sf,bf}, margin=1cm]{caption}

\newtheorem{theorem}{Theorem}[section]
\newtheorem{proposition}[theorem]{Proposition}
\newtheorem{lemma}[theorem]{Lemma}

\newtheorem{notation}[theorem]{Notation}
\theoremstyle{remark}
\newtheorem{remark}[theorem]{Remark}

\newcommand{\R}{\mathbb{R}}
\newcommand{\N}{\mathbb{N}}
\newcommand{\Z}{\mathbb{Z}}

\newcommand{\C}{\mathbb{C}}
\newcommand{\D}{\mathbb{D}}
\newcommand{\Dc}{\mathcal{D}}
\newcommand{\Ec}{\mathcal{E}}
\newcommand{\Fc}{\mathcal{F}}
\newcommand{\Nc}{\mathcal{N}}
\newcommand{\Hc}{\mathcal{H}}
\newcommand{\Tc}{\mathcal{T}}
\newcommand{\xf}{\mathbf{x}}
\newcommand{\yf}{\mathbf{y}}
\newcommand{\zf}{\mathbf{z}}

\newcommand{\Expect}[1]{\mathbb{E} \left[ #1 \right] }
\newcommand{\prob}{\mathbb{P}}
\newcommand{\Prob}[1]{\mathbb{P} \left( #1 \right) }

\newcommand{\abs}[1]{\left\vert #1 \right\vert}
\newcommand{\norme}[1]{\left\| #1 \right\| }
\newcommand{\scalar}[1]{\left\langle #1 \right\rangle }
\newcommand{\floor}[1]{\left\lfloor #1 \right\rfloor}
\newcommand{\indic}[1]{ \mathbf{1}_{ \left\{ #1 \right\} } }
\newcommand{\eps}{\varepsilon}
\DeclareMathOperator{\CR}{CR}
\DeclareMathOperator{\loc}{loc}

\newcommand{\E}{\mathbb{E}}
\newcommand{\mc}[1]{\mathcal{#1}}

\renewcommand{\P}{\mathbb{P}}
\newcommand{\reals}{\mathbb{R}}

\newcommand{\integers}{\mathbb{Z}}

\newcommand{\1}{\mathbf{1}}
\renewcommand{\S}{\mathcal{S}}
\DeclareMathOperator{\supp}{supp}

\renewcommand{\Re}{\operatorname{Re}}
\renewcommand{\Im}{\operatorname{Im}}

\title{Noise-like analytic properties of imaginary chaos}
\author{Juhan Aru}
\address{École Polytechnique Fédérale de Lausanne, Lausanne, Switzerland}
\email{juhan.aru@epfl.ch}

\author{Guillaume Baverez}
\address{Institut de Math\'ematiques de Marseille, Marseille, France}
\email{guillaume.baverez@univ-amu.fr}

\author{Antoine Jego}
\address{CNRS \& CEREMADE, Université Paris-Dauphine, PSL University, France}
\email{antoine.jego@dauphine.psl.eu}

\author{Janne Junnila}
\address{University of Helsinki, Helsinki, Finland}
\email{janne.junnila@helsinki.fi}

\date {}

\numberwithin{equation}{section}

\begin{document}

\renewcommand{\theparagraph}{\thesubsection.\arabic{paragraph}} 

\begin{abstract}
In this note we continue the study of imaginary multiplicative chaos $\mu_\beta := \exp(i \beta \Gamma)$, where $\Gamma$ is a two-dimensional continuum Gaussian free field. We concentrate here on the fine-scale analytic properties of $|\mu_\beta(Q(x,r))|$ as $r \to 0$, where $Q(x,r)$ is a square of side-length $2r$ centred at $x$. More precisely, we prove monofractality of this process, a law of the iterated logarithm as $r \to 0$ and analyse its exceptional points, which have a close connection to fast points of Brownian motion. Some of the technical ideas developed to address these questions also help us pin down the exact Besov regularity of imaginary chaos, a question left open in \cite{JSW}. All the mentioned properties illustrate the noise-like behaviour of the imaginary chaos. We conclude by proving that the processes $x \mapsto |\mu_\beta(Q(x,r))|^2$, when normalised additively and multiplicatively, converge as $r \to 0$ in law, but not in probability, to white noise; this suggests that all the information of the multiplicative chaos is contained in the angular parts of $\mu_\beta(Q(x,r))$.     
\end{abstract}

\maketitle

\setcounter{tocdepth}{1}
\tableofcontents

\section{Introduction and main results}

Let $D \subset \R^2$ be an open bounded simply connected domain and let $\Gamma$ be the Gaussian free field in $D$ with zero boundary condition. That is, $\Gamma$ is a Gaussian, centred Schwartz distribution whose covariance is given by the Green function $G_D$ of the Laplacian with Dirichlet boundary condition, in the sense that for any smooth test functions $\varphi$, $\psi$,
\[
\E[(\Gamma,\varphi) (\Gamma, \psi)] = \int_{D \times D} \varphi(x) G_D(x,y) \psi(y) dx \, dy.
\]
In this paper, we normalise the Green function so that $G_D(x,y) \sim -\log \abs{x-y}$ when $x - y \to 0$. For $\beta \in (0,\sqrt{2})$ the imaginary chaos $\mu(x)$ is formally defined as
$e^{i \beta \Gamma(x) } \indic{x \in D}.$
Because $\Gamma$ is not defined pointwise, the exponential of $i \beta \Gamma$ is \textit{a priori} not well defined. A simple regularisation procedure helps out: for $\eps >0$, let $\Gamma_\eps(x)$ be the $\eps$-circle average of $\Gamma$ about $x \in D$; then $\mu$ is defined as the limit
\begin{equation}
    \label{E:defmu}
\mu(x) \qquad \coloneqq \qquad \lim_{\eps \to 0} \eps^{-\beta^2/2} e^{i \beta \Gamma_\eps(x)}  \indic{x \in D}
\end{equation}
where the limit takes place in $L^2$ in any negative Sobolev space $H^s$ with $s <-1$; see \cite[Section 3.1]{JSW}. In particular, $(\mu, \phi)$ are well-defined random variables simultaneously for all smooth test functions $\phi$. Further, if $f = \mathbf{1}_A$ is the indicator function of a Borel set $A \subset D$ (or indeed any measurable bounded function), one can define $\mu(A) = (\mu, \mathbf{1}_A)$ as the $L^2$ limit of $\eps^{-\beta^2/2} \int_D f e^{i \beta \Gamma_\eps}$.

Imaginary multiplicative chaos seems to have first come up in the context of directed polymers \cite{derrida93}, though now we could name it rather the total mass of imaginary multiplicative cascade. Imaginary multiplicative cascades, i.e. imaginary chaos for log-correlated fields with a hierarchical structure, were studied in \cite{barral2010b} and subsequently in \cite{barral2010a, barral2010c}. A generalization to more general log-correlated fields came about in the context of complex chaos in \cite{complexGMC}. In \cite{JSW} the authors not only studied imaginary chaos as an interesting analytic object, but made a direct connection to the scaling limit of spins of the critical XOR-Ising model. Thereafter imaginary chaos has been studied in several papers - in \cite{aru2022density, MR4814241} it was shown that $(\mu, f)$ has a bounded smooth positive density on $\C$, in \cite{ajreconstruct} it was shown that imaginary chaos of the free field - in contrast to imaginary cascades - contains all the information of the underlying log-correlated field. Finally, due to its connections to Coloumb gases, imaginary chaos is now making appearance also in probabilistic constructions of 2D conformal field theories \cite{gkrimaginary}.

In this note we continue the study of imaginary multiplicative chaos. 
We concentrate on the study of the process $(|\mu(Q(x,r))|)_{x \in D,r > 0}$, where $Q(x,r) = x + [-r,r]^2$ are squares of side length $2r$ centred at $x$ (squares are considered instead of discs so that we can refer to a result of \cite{Leble2017}; see Theorem \ref{T:LSZ} below).
We will see that this process has a very noise-like character that does not exhibit in the $r \to 0$ limit any non trivial macroscopic properties. Notice that for example the real part of $\mu(Q(x,r))$ does give rise to the interesting macroscopic structure (\cite{JSW}), as probably does its angle. Imaginary multiplicative chaos is considerably harder to study than the real chaos due to the lack of monotonicity w.r.t. underlying field, and the cancellations that need to be taken into account carefully. For example, to prove the existence of density, one can use a simple Fourier argument in the case of real chaos, but one relies on Malliavin calculus in the case of imaginary chaos. In this note, we again need to combine (functional) analytic tools with probabilistic arguments; several small technical ideas and interplays arise, which we think might be of interest in studying random fields of similar nature. 

We will next explain what is contained in the paper in terms of results, and also highlight some of the technical inputs going into proving these results.

\subsection{Monofractality}

Our first main result proves a certain monofractality of the imaginary chaos $\mu$. To state this result, we first show that the process $(\mu(Q(x,r)))_{x \in D, r \ge 0}$ has a version which is continuous. This is \textit{a priori} not clear since $\mu$ is not a random measure; see \cite[Theorem 1.2]{JSW}.

\begin{proposition}[Hölder regularity]\label{prop:Holder regularity}
The map
\[
\begin{array}{ccc}
\{ (x,r) \in D \times [0,1], Q(x,r) \subset D \} & \longrightarrow & \C \\
(x,r) & \longmapsto & \mu(Q(x,r))
\end{array}
\]
possesses a continuous modification. Moreover, for all compact subsets $K$ of $D$ and $\alpha < 1 - \beta^2/4$, there exists a random constant $C_{K,\alpha}$, almost surely finite, such that for all $Q(x,r), Q(x',r') \subset K$,
\begin{equation}
\label{eq:prop_Holder}
\abs{\mu(Q(x,r)) - \mu(Q(x',r'))} \leq C_{K,\alpha} (r \vee r')^\alpha \norme{ (x,r) - (x',r')}^\alpha.
\end{equation}
\end{proposition}

As in the real case, also the imaginary chaos is believed to share a lot of properties with complex multiplicative cascades which have been thoroughly studied in \cite{barral2010a,barral2010b,barral2010c}. In particular, if we restrict ourselves to cascades built using complex exponentials of a normal distribution, \cite[Theorem 2.9]{barral2010b} establishes the monofractality of the cascade. Our first theorem claims that the same result holds for the imaginary chaos $\mu$.

\begin{theorem}[Monofractality]\label{th:monofractality}
With probability one, simultaneously for all $x \in D$,
\begin{equation}\label{eq:th_monofractality}
\liminf_{r \to 0} \frac{\log \abs{\mu(Q(x,r))}}{ \log r} = 2 - \beta^2/2.
\end{equation}
\end{theorem}

A key technical input here is the existence and boundedness of density for imaginary chaos, i.e. the density of variables of the type $(\mu, f)$ where $f$ is a test function, first proved in \cite{aru2022density} using Malliavin calculus. In Theorem \ref{T:density} we revisit the original proof to bring out some quantitative bounds on the density that help us treat several simultaneous test functions at once.


We remark that the monofractality result contrasts sharply with the case of a two-dimensional real multiplicative $M_\gamma$ defined by
\[
M_\gamma(dx) = \lim_{\eps \to 0} \eps^{\gamma^2/2} e^{\gamma \Gamma_\eps(x)} dx
\]
for $\gamma \in (0,2)$. Indeed, the quantity 
\[
\liminf_{r \to 0} \frac{\log M_\gamma(Q(x,r))}{\log r}
\]
is known to depend heavily on the point $x$ and takes values in $[ ( \sqrt{2} - \gamma/\sqrt{2} )^2, ( \sqrt{2} + \gamma/\sqrt{2} )^2 ]$  (see \cite[Section 4]{rhodes2014} for instance).

\begin{remark}\label{rem:exact scale invariance}
To get some intuition on Theorem \ref{th:monofractality} and its difference with respect to real chaos, we can consider near the origin the field $X$ with correlations given by $\Expect{X(x)X(y)} = - \log \abs{x-y}$. In that case, we can decompose $\Expect{X(x)X(y)} = -\log (\abs{x-y} /r) + \log (1/r)$ so that $X$ is the independent sum $X = X_r + N_r$ with
\[
(X_r(rx), x \in \R^2) \overset{\mathrm{(d)}}{=} (X(x),x \in \R^2)
\qquad \mathrm{and} \qquad
N_r \sim \Nc(0,\log (1/r)).
\]
In particular,
\begin{align*}
\int_{Q(0,r)} e^{i \beta X(y)} dy = r^{2-\beta^2/2} e^{i \beta N_r} \int_{Q(0,1)} e^{i \beta X_r(ry) } dy,
\end{align*}
where the $r^{-\beta^2/2}$ comes from the renormalisation of $e^{i \beta X}$.
Therefore, in the case of imaginary chaos, the distribution of $r^{-2+\beta^2/2} \abs{\int_{Q(0,r)} e^{i \beta X(y)} dy}$ does not depend on $r$. In the case of real chaos, on the contrary, 
the term $e^{\gamma N_r}$ does not simplify and can be seen to be at the root of the multifractal behaviour.
\end{remark}

\subsection{Law of the iterated logarithm and fast points}

We next detail the study of the local behaviour of $\mu$ by providing a law of the iterated logarithm for $\mu$.

\begin{theorem}[Law of the iterated logarithm]\label{th:LIL}
There exists $c^*(\beta)>0$ depending only on $\beta$ (see Theorem~\ref{T:LSZ} below) such that for any given $x \in D$, we have
\begin{equation}
\label{eq:th_LIL}
\limsup_{r \to 0} \frac{\abs{\mu(Q(x,r))}}{r^{2-\beta^2/2} (\log \abs{ \log r} )^{\beta^2/4}} = c^*(\beta)^{-\beta^2/4} \quad \quad \mathrm{a.s.}
\end{equation}
\end{theorem}
To prove this result, we again turn to the analytic toolbox and start off by making use of Sobolev embeddings to swap a uniform norm $\norme{\cdot}_\infty$ for a more tractable one. See the beginning of Section \ref{SS:sup_summable} for more details.

As a next step we study exceptional points where $\mu$ has an atypical behaviour and breaks the law of the iterated logarithm. These points are an analogue of the fast points of Brownian motion \cite[Section 10.1]{morters_peres_2010} - indeed, if one thinks of imaginary chaos as a noise-field,
then integrating it over squares would exactly give a 2D analogue for the definition of fast points. Whereas different in nature, in terms of being atypical, one can also recall here the notion of thick points of the GFF $\Gamma$ \cite{HuMillerPeres2010}.

In the following, we will denote by $\dim_\Hc(A)$ the Hausdorff dimension of a set $A \subset \C$.

\begin{theorem}[Exceptional points]\label{th:exceptional points}
We have
\begin{equation}
\label{eq:th_exceptional_sup}
\sup_{x \in D} \limsup_{r \to 0} \frac{ \abs{\mu(Q(x,r))} }{r^{2-\beta^2/2} \abs{\log r}^{\beta^2/4}} = 2^{\beta^2/4} c^*(\beta)^{-\beta^2/4} \quad \quad \mathrm{a.s.}
\end{equation}
Moreover, if we define for $a>0$ the set of $a$-fast points by
\begin{equation}
\Tc(a) \coloneqq \left\{ x \in D: \limsup_{r \to 0} \frac{ \abs{\mu(Q(x,r))} }{r^{2-\beta^2/2} \abs{\log r}^{\beta^2/4}} \geq a \right\},
\end{equation}
then we have for all $0\leq a\leq\left( 2/c^*(\beta) \right)^{\beta^2/4}$,
\begin{equation}
\label{eq:th_exceptional_number}
\dim_\Hc \left( \Tc(a) \right)
= 2 - c^*(\beta) a^{4/\beta^2}
\quad \quad \mathrm{a.s.}
\end{equation}
\end{theorem}
\medskip 

\noindent \textbf{The constant \texorpdfstring{$c^*(\beta)$}{c*(beta)}} appearing in Theorems \ref{th:LIL} and \ref{th:exceptional points} comes from the asymptotics of the tails of the imaginary chaos derived in \cite[Corollary A.2]{Leble2017}. We now state this result precisely.
Consider the whole-plane GFF $\Gamma^*$ and its associated imaginary chaos $\mu^*(y) = \lim_{\eps \to 0} \eps^{-\beta^2/2} e^{i\beta \Gamma^*_\eps(y)} $, where $\Gamma_\eps^*$ is the $\eps$-regularisation of $\Gamma^*$ using circles averages. The whole-plane GFF $\Gamma^*$ is a random generalised function modulo global additive constants. These additive constants disappear when one considers $\abs{ \int_{Q(0,1)} \mu^*(y) dy }$ and so this latter object is a well-defined nonnegative real random variable.

\begin{theorem}[\cite{Leble2017}]\label{T:LSZ}
There exists a constant $c^*(\beta)>0$ which may depend on $\beta$ such that
\begin{equation}\label{eq:rem tail}
\Prob{ \abs{ \int_{Q(0,1)} \mu^*(y) dy } > t } = \exp \left( - c^*(\beta) t^{4/\beta^2} + o \left( t^{4/\beta^2} \right) \right)
\mathrm{~as~} t \to \infty.
\end{equation}
\end{theorem}

The main result of \cite{Leble2017} (Theorem 1 and Corollary 1.3 therein) actually provides extremely precise asymptotics for the moments of $\abs{ \int_{Q(0,1)} \mu^*(y) dy }^2$ (which are closely related to 2D two-component plasma). The estimate \eqref{eq:rem tail} is then essentially shown to be a consequence of their moment estimates; see Appendix \ref{app:tail} for further comments and an alternative derivation.

\subsection{Exact Besov regularity of imaginary chaos}
\medskip

The fractal properties of imaginary chaos are tightly linked to its Besov regularity, a fact that can be interpreted in the broader context of the so-called ``Frisch-Parisi conjecture'' which relates more generally multifractal properties with Besov regularity (see e.g. \cite{Jaffard00} for an account).

In \cite[Theorem 3.16]{JSW}, it was shown that for all $1\leq p,q\leq\infty$ and $s\in\R$, it a.s. holds that $\mu$ belongs to the Besov space $B^s_{p,q}(D)$ if $s<-\frac{\beta^2}{2}$ and $\mu\not\in B^s_{p,q}(D)$ if $s>-\frac{\beta^2}{2}$. In this note, using ideas that we developed to prove Theorems \ref{th:monofractality} and \ref{th:LIL}, we can settle the critical case $s=-\frac{\beta^2}{2}$, i.e. we can decide for which values of $p,q\in[1,\infty]$ it holds that $\mu\in B^{-\frac{\beta^2}{2}}_{p,q}(D)$. We also note that these values are the same as the ones found for the regularity of white noise \cite{Veraar10}. This is consistent with the fact that the imaginary chaos converges to white noise when $\beta \to \sqrt{2}$ \cite{JSW}.

\begin{theorem}[Regularity of imaginary chaos]\label{thm:regularity}
Let $1\leq p,q\leq\infty$. 
If $p< \infty$ and $q= \infty$, then $\mu\in B^{-\frac{\beta^2}{2}}_{p,q,\loc}(D)$ a.s. Otherwise, $\mu \not\in B^{-\frac{\beta^2}{2}}_{p,q,\loc}(D)$ a.s.
\end{theorem}

\subsection{The limiting field of $(|\mu(Q(x,r))|)_{x,r}$}

\noindent We have seen that the absolute value of the imaginary chaos is monofractal. Here we ask the question of whether any interesting macroscopic information is left as $r \to 0$. For technical convenience, we work with the field $ |\mu(Q(x,r)|^2$ instead, and assume without loss of generality that $D$ contains $[0,1]^2 + B(0,2)$.

First, just because of scaling, this field, converges to $0$ with no extra normalization. Further, when one normalizes by $\delta^{\beta^2 - 4}$, it converges to a constant field (this will actually follow from the proof below) and thus contains no interesting information. A natural question is whether one could regularize the process in a different manner and obtain an interesting limiting field. One option would be to look for instance at its oscillations around its mean. 

We will show that, by subtracting the mean and renormalizing well, one indeed gets a non-trivial limit; however, the limiting field turns out to be just white noise. As the convergence takes place only in law and not in probability, it seems plausible to say that in the limit most of the information on the underlying GFF is lost. In fact, following the proof below, one could make it even precise, by showing that the same result also holds conditioned on any finite number of Fourier frequencies.

\begin{theorem}\label{thm:whitenoise2}
  The processes
  \begin{equation}
      \label{E:Wr}
  W_r := (r^{\beta^2-5}(|\mu(Q(x,r))|^2 - \E(|\mu(Q(x,r))|^2)))_{x \in [0,1]^2}
  \end{equation}
  converge as $r \to 0$ in distribution, but not in probability, in the space $H^{-2}_0([0,1]^2)$, to a constant multiple of the White noise in $[0,1]^2$.
\end{theorem}


Although the path towards a proof is rather clear, given the independence properties noticed in earlier parts of the paper, it is technically non-trivial to push through. For example some symmetrization tricks are needed to identify cancellations and obtain precise enough estimates; see the paragraph following \eqref{E:def_A_symmetrization} for more details.

Finally, let us also mention a related result but with a different flavour from \cite{JSW}: when $\beta \to \sqrt{2}$, one can renormalise the imaginary chaos so that it converges to white noise.

\subsection{Organisation of the paper}

The paper is organised as follows:

\begin{itemize}
    \item Section \ref{S:preliminaries} is a preliminary section recalling a few facts about imaginary multiplicative chaos. We also prove two lemmas that will be useful in the rest of the paper (Lemmas~\ref{lem:delete harmonic} and \ref{lem:tail}).
    \item Section \ref{S:monofractality} is dedicated to the study of the monofractality of $\mu$ and contains the proofs of Proposition \ref{prop:Holder regularity} and Theorem \ref{th:monofractality}.
    \item Section \ref{S:LIL} is where we prove the law of the iterated logarithm (Theorem \ref{th:LIL}).
    \item Section \ref{S:exceptional} is concerned with the study of the fast points and proves Theorem \ref{th:exceptional points}.
    \item Section \ref{sec:Besov} is where we prove the optimal Besov regularity of $\mu$, i.e. Theorem \ref{thm:regularity}.
    \item Section \ref{S:WN} is finally dedicated to the the proof of Theorem \ref{thm:whitenoise2}: the convergence of $W_r$ to white noise.
    \item Appendix \ref{app:kolmogorov}: we state and prove a variant of Kolmogorov's continuity theorem that we use in the proof of Proposition \ref{prop:Holder regularity}.
    \item Appendix \ref{app:tail} contains a few comments on the tail estimate \eqref{eq:rem tail}.
\end{itemize}

\medskip

\noindent\textbf{Acknowledgements}
The authors thank Nathana\"el Berestycki for interesting discussions during the very early stages of the project and for the referees for their very careful reading of the paper. J.A is supported by Eccellenza grant 194648 of the Swiss National Science Foundation and is a member of NCCR Swissmap. G.B is supported by ANR-21-CE40-0003. A.J. was a member of NCCR Swissmap. J.J is supported by The Finnish Centre of Excellence (CoE) in Randomness and Structures and was a member of NCCR Swissmap.

\section{Preliminaries}\label{S:preliminaries}

\subsection{Notational convention:}
In this article, we will use two different normalisations for the exponential of $i \beta \Gamma$:
\begin{equation}
\label{E:normalisation_iGMC1}
e^{i \beta \Gamma(x)} \quad = \quad \lim_{\eps \to 0} \eps^{-\beta^2/2} e^{i \beta \Gamma_\eps(x)}
\qquad \text{and} \qquad
:e^{i \beta \Gamma(x)}: \quad = \quad \lim_{\eps \to 0} e^{\beta^2 \mathrm{Var}(\Gamma_\eps(x))/2} e^{i \beta \Gamma_\eps(x)},
\end{equation}
where $\Gamma_\eps$ stands for the $\eps$-circle average of $\Gamma$. The second normalisation is the so-called Wick normalisation and is distinguished from the first one by the two colons on both sides of the exponential.
These two normalisations lead to equivalent generalised functions in the sense that they are related by
\begin{equation}
\label{E:normalisation_iGMC2}
e^{i \beta \Gamma(x)} \quad = \quad \CR(x,D)^{-\beta^2/2} :e^{i \beta \Gamma(x)}:.
\end{equation}
These normalisations have their own advantages. For instance, $e^{i \beta \Gamma}$ has a neat spatial Markov property, whereas the moments of $:e^{i \beta \Gamma}:$ have a simpler expression.

\subsection{General moment bounds for imaginary chaos}

We start this preliminary section by collecting some known estimates on the moments of the imaginary chaos. Consider a log-correlated field $X$ defined on some 2D open set $D \subset \R^2$. To be more specific and to be able to apply the existing literature, we assume that the covariance of $X$ can be written as
\[
-\log|x-y| + g(x,y)
\]
where $g \in C^2(D \times D) \times L^1(D \times D)$ is bounded from above.
For any $\mathbf{x} = (x_1, \dots, x_N) \in D^N$ and $\mathbf{y} = (y_1, \dots, y_M) \in D^M$, denote by
\begin{equation}
\label{E:defEcal}
\Ec(X;\mathbf{x},\mathbf{y}) := -\sum_{1 \leq i < j \leq N} \E[X(x_i)X(x_j)] -\sum_{1 \leq i < j \leq M} \E[X(y_i)X(y_j)] + \sum_{\substack{1 \leq i \leq N\\1 \leq j \leq M}} \E[X(x_i)X(y_j)].
\end{equation}
We will write $d \mathbf{x} = dx_1 \dots dx_N$.

\begin{lemma}\label{L:moment_JSW}
For any $f \in L^\infty(D)$,
\[
\Expect{\abs{\int :e^{i \beta X(x)}: f(x) dx}^{2N}}
= \int_{D^N \times D^N} \,d\mathbf{x} \,d\mathbf{y} \, e^{\beta^2 \Ec(X;\mathbf{x};\mathbf{y})} \prod_{i=1}^N f(x_i) \overline{f(y_i)}.
\]
\end{lemma}

\begin{proof}
See \cite[Section 3.2]{JSW}.
\end{proof}

To bound the moments, we will moreover use the following Onsager inequality:

\begin{lemma}[\cite{JSW}, Proposition 3.6 (i)]\label{L:bound_Ec}
Let $K$ be a compact subset of $D$ and $(\mathbf{x}, \mathbf{y}) \in D^N \times D^N$. Denote by $z_i = x_i$ and $z_{N+i} = y_i$ for $i=1, \dots, N$. Then
\[
\Ec(X;\xf;\yf) \leq - \frac{1}{2} \sum_{j=1}^{2N} \log \min_{i \neq j} |z_i - z_j| + CN
\]
for some constant $C>0$ depending only on $g$ and $K$.
\end{lemma}

\begin{lemma}[\cite{JSW}, Lemma 3.10]\label{L:integral}
Let $B(0,1)$ be the unit ball in $\R^2$. For all $\beta \in (0,\sqrt{2})$, there exists $C=C(\beta,d)>0$ such that for all $p \geq 1$,
\[
\int_{B(0,1)^p} d \zf \exp \Big( - \frac{\beta^2}{2} \sum_{j=1}^p \log \min_{i \neq j} |z_i - z_j| \Big) \leq C^p p^{\beta^2 p/4}.
\]
\end{lemma}

\begin{notation}
For $x, y \in D$, we will denote by
\begin{equation}
    \label{E:def_GD'}
G_D'(x,y) = G_D(x,y) - \frac{1}{2}\log \CR(x,D) - \frac12 \log \CR(y,D).
\end{equation}
For any $\mathbf{x} = (x_1, \dots, x_N) \in D^N$ and $\mathbf{y} = (y_1, \dots, y_M) \in D^M$, denote by
\begin{equation}
\label{E:defEcal'}
\Ec'(\Gamma;\mathbf{x},\mathbf{y}) := -\sum_{1 \leq i < j \leq N} G_D'(x_i,x_j) -\sum_{1 \leq i < j \leq M} G_D'(y_i,y_j) + \sum_{\substack{1 \leq i \leq N\\1 \leq j \leq M}} G_D'(x_i,y_j).
\end{equation}
By Lemma \ref{L:moment_JSW} and the difference of normalisation \eqref{E:normalisation_iGMC2} between $\mu$ and the Wick normalisation of $e^{i \beta \Gamma}$, we have for all test function $f : D \to \R$,
\begin{equation}
    \label{E:moment_iGFF2}
\E \abs{\int_ D \mu(x) f(x) \, dx}^2 = \int_{D \times D} e^{\beta^2 G_D'(x,y)} f(x) f(y) \,dx \,dy
\end{equation}
and, more generally, for all $N \geq 1$,
\begin{equation}
    \label{E:moment_iGFF}
\E \abs{\int_ D \mu(x) f(x) \, dx}^{2N} = \int_{D^N \times D^N}
e^{\beta^2 \Ec'(\Gamma;\xf;\yf)} \prod_{j=1}^N f(x_j) f(y_j) \,d \xf \,d \yf.
\end{equation}
\end{notation}

\subsection{Moment bounds for rectangles}

We will also need several times to control how the moments of imaginary chaos depend on the side-lengths of a rectangle. This is given by the following lemma:

\begin{lemma}\label{lemma:rectangle1}
Let $K$ be a compact subset of $D$. There exists $C = C(K,\beta)$ such that the following holds. For any $a,b \in (0,1]$, any $a \times b$ rectangle $Q \subset K$ and any $N \geq 1$, we have
  \[\E | \mu(Q) |^{2N} \le C^N (a b)^{(2 - \frac{\beta^2}{2})N}  N^{\frac{\beta^2}{2}N}.\]
\end{lemma}

Before proving Lemma \ref{lemma:rectangle1}, we first recall the following recursive inequality which works for any log-correlated field:

\begin{lemma}\label{lemma:recursion1}
Let $\nu$ be an imaginary chaos defined on some domain $D$ (with the Wick normalisation) and $A \subset D$ be a measurable set.
Let $m \geq 1$ and $A_1,\dots,A_m$ be disjoint measurable subsets of $A$. Then for all $N_1, \dots, N_m \geq 1$, we have
  \[\frac{1}{(N!)^2}\E |\nu(A)|^{2N} \ge \prod_{j=1}^m \frac{1}{(N_j!)^2} \E |\nu(A_j)|^{2N_j}, \quad \text{where} \quad N = N_1 + \dots + N_m.
  \]
\end{lemma}

\begin{proof}
  See the proof of \cite[Proposition~3.14]{JSW}.
\end{proof}


\begin{proof}[Proof of Lemma \ref{lemma:rectangle1}]
It will be useful for us to compare the field $\Gamma_D$ in $K$ with a stationary log-correlated field.
Such a comparison is given by \cite[Theorem~4.5]{aru2022density}, which allows us to write in $K$ the decomposition
\begin{equation}\label{eq:decomposition1}
  \Gamma_D = Z + Y,
\end{equation}
where $Z$ and $Y$ are two independent fields with the following properties.
The field $Y$ is Hölder-regular.
On the other hand,
$Z$ is a field defined on the entire space $\R^2$ whose covariance is of the form
\begin{equation}
    \label{E:covarianceZ}
\E Z(x)Z(y) = \int_0^{\infty} k(e^u(x - y))(1 - e^{-\delta u}) \, du
\end{equation}
for some $\delta > 0$ and some rotationally invariant and compactly supported seed covariance $k \colon \reals^2 \to \reals$ with $k(0) = 1$ (see \cite[Assumption 4.2]{aru2022density} for the list of properties satisfied by $k$). 
In this proof, the property that will be the most useful for us is the fact that the law of $Z$ is invariant under translations.

Let $Q$ be a rectangle included in $K$. Let us first show that
\begin{equation}
\label{E:rectangle_claim}
\E |\mu(Q)^{2N}| \le C^N \E \big| \int_Q :e^{i \beta Z(x)}: \, dx \big|^{2N}.
\end{equation}
By independence of $Z$ and $Y$ and by Lemma \ref{L:moment_JSW},
  \begin{align*}
  \E |\mu(Q)^{2N}| & = \int_{Q^N \times Q^N} d\xf \,d\yf e^{\beta^2 \Ec(Z;\xf;\yf)} \prod_{j=1}^N \CR(x_j,D)^{-\frac{\beta^2}{2}} \CR(y_j,D)^{-\frac{\beta^2}{2}} \\
  & \qquad \times \Expect{ e^{i \beta Y(x_1)} \dots e^{i \beta Y(x_N)} e^{-i \beta Y(y_1)} \dots e^{-i \beta Y(y_N)} }. 
  \end{align*}
  By Jensen's inequality, the expectation in the integrand is at most 1. The remaining terms being nonnegative and $\sup_{x \in K} \CR(x,D)^{-\beta^2/2}$ being finite, we deduce that $\E |\mu(Q)^{2N}|$ is at most
  \[
  C^N \int_{Q^N \times Q^N} d\xf \,d\yf e^{\beta^2 \Ec(Z;\xf;\yf)}
  = C^N \E \big| \int_Q :e^{i \beta Z(x)}: \, dx \big|^{2N},
  \]
  as stated in \eqref{E:rectangle_claim}.

It is now enough to show that for any $a \times b$ rectangle $Q \subset \R^2$, 
\begin{equation}
\label{E:rectangle_claim2}
\E \big| \int_Q :e^{i \beta Z(x)}: \, dx \big|^{2N}
\leq C^N (a b)^{(2 - \frac{\beta^2}{2})N}  N^{\frac{\beta^2}{2}N}.
\end{equation}
The square $[0,1]^2$ contains $m = \lfloor \frac{1}{a} \rfloor \lfloor \frac{1}{b} \rfloor$ disjoint rectangles of size $a \times b$.
  Then by Lemma~\ref{lemma:recursion1} and stationarity we have
  \[\frac{1}{((m N)!)^2} \E \Big| \int_{[0,1]^2} :e^{i \beta Z(x)}: \, dx \Big|^{2m N} \ge \Big( \frac{1}{(N!)^2} \E \Big| \int_Q :e^{i \beta Z(x)}: \, dx\Big|^{2N} \Big)^m,\]
  which gives us
  \begin{equation}
    \label{E:pf_rectangle}
  \E \Big| \int_Q :e^{i \beta Z(x)}: \, dx \Big|^{2N} \le \frac{(N!)^2}{((mN)!)^{2/m}} \Big(\E \Big|\int_{[0,1]^2} :e^{i \beta Z(x)}: \,d x\Big|^{2Nm}\Big)^{1/m}.
  \end{equation}
  By applying \cite[Theorem~1.3]{JSW} (or, equivalently, Lemmas \ref{L:moment_JSW}, \ref{L:bound_Ec} and \ref{L:integral})  we have
  \[\Big(\E \Big|\int_{[0,1]^2} :e^{i \beta Z(x)}: \,d x\Big|^{2Nm}\Big)^{1/m} \le C^N (Nm)^{\frac{\beta^2}{2} N}\]
  for some constant $C>0$.
  One easily checks by using Stirling's formula that the remaining multiplicative factor on the right hand side of \eqref{E:pf_rectangle} is at most $C^N m^{-2N}$, concluding the proof of \eqref{E:rectangle_claim2}.
This finishes the proof of Lemma \ref{lemma:rectangle1}.
\end{proof}

\subsection{Contribution of the harmonic extension to imaginary chaos}

The proofs of some of our main results are based on the spatial Markov property which allows us to find a decomposition analogous to the one presented in Remark \ref{rem:exact scale invariance}. The difference here is that the field $N_r$ corresponding to the harmonic part of the GFF is no longer constant. However, as it does not vary too rapidly we can recover an argument similar to Remark \ref{rem:exact scale invariance}. This is the content of the next lemma. We will need this lemma for both the GFF $\Gamma$ in $D$ with Dirichlet boundary conditions and for the whole-plane GFF $\Gamma^*$ (see \cite[Proposition 2.8]{ImaginaryGeometryIV} for the spatial Markov property for $\Gamma^*$).

\begin{lemma}\label{lem:delete harmonic}
Let $x \in D$ and $R_0 >0$ such that $Q(x,R_0) \subset D$. Let $0 < 2r < R < R_0$. We decompose the GFF $\Gamma$ in $D$ and the whole-plane GFF $\Gamma^*$ as the independent sums
\[
\Gamma = \Gamma^{(R)} + h^{(R)}
\quad \mathrm{~and~} \quad
\Gamma^* = \Gamma^{(R)} + {h^*}^{(R)}
\]
where $\Gamma^{(R)}$ is a GFF in $Q(x,R)$ and $h^{(R)}$ and ${h^*}^{(R)}$ are harmonic in $Q(x,R)$. By denoting $\mu^{(R)}$ the imaginary chaos associated to $\Gamma^{(R)}$, there exist $c,c_{R_0}>0$ which may depend on $\beta$ and on $\beta, R_0$ respectively such that for all $\lambda \geq (r/R)^{1/2}$,
\begin{align}
\label{eq:lem delete harmonic}
& \Prob{ \abs{ \abs{ \mu(Q(x,r)) } - \abs{ \mu^{(R)}(Q(x,r)) } } > \lambda r^{2-\frac{\beta^2}{2}}  } \leq \frac{1}{c_{R_0}} \exp \Big( - c_{R_0} \Big( \frac{R}{r} \lambda^2 \Big)^{ \frac{2}{\beta^2+1} } \Big),
\end{align}
\begin{align}
\label{eq:lem delete harmonic2}
\Prob{ \abs{ \abs{ \mu^*(Q(x,r)) } - \abs{ \mu^{(R)}(Q(x,r)) } } > \lambda r^{2-\frac{\beta^2}{2}}  } \leq \frac{1}{c} \exp \Big( - c \Big( \frac{R}{r} \lambda^2 \Big)^{ \frac{2}{\beta^2+1} } \Big).
\end{align}
\end{lemma}


\begin{proof}[Proof of Lemma \ref{lem:delete harmonic}]
We start by proving \eqref{eq:lem delete harmonic} whose proof is divided into two main steps. Firstly, we will give quantitative estimates saying that $h^{(R)}$ does not vary too much on $Q(x,r)$. We will then use Markov's inequality and compute moments of integrals of the imaginary chaos. This has the advantage of transforming oscillating integrals into integrals of nonnegative functions which are easier to handle.

We start by proving the following claim: there exists $C_{R_0}>0$ which may depend on $R_0$ s.t.
\begin{equation}
\label{eq:proof_harmonic_covariance}
\sup_{y \in Q(x,r)} \Expect{ \left( h^{(R)}(y)-h^{(R)}(x) \right)^2}
\leq C_{R_0} \frac{r}{R}.
\end{equation}
An exact computation shows that for $y,z \in Q(x,r)$,
\begin{equation}
\label{eq:proof_harmonic_covariance1}
\Expect{\left( h^{(R)}(y) - h^{(R)}(x) \right)^2} = 2 G_{Q(x,R)}(x,y) - 2 G_D(x,y) + \log \frac{\CR(x,D) \CR(y,D)}{\CR(x,Q(x,R)) \CR(y,Q(x,R))}.
\end{equation}
Let $\varphi : D \to \D$ be a conformal map. By conformal invariance of the Green function,
\begin{align*}
& 2G_D(x,y) - \log (\CR(x,D) \CR(y,D)) \\
& = 2 \log \frac{\abs{1 - \overline{\varphi(x)}\varphi(y)}}{\abs{\varphi(x) - \varphi(y)}} + \log \left( \abs{\varphi'(x)} \abs{\varphi'(y)} \right) - \log \left( 1 - \abs{\varphi(x)}^2 \right) \left( 1 - \abs{\varphi(y)}^2 \right).
\end{align*}
In the following, we will make computations which are justified because $\varphi \in C^\infty(D)$ and $\varphi^{-1} \in C^\infty(\D)$. The error terms will thus depend on $R_0$.
We have
\[
\log \abs{ \varphi(x) - \varphi(y)} = \log \abs{\varphi'(x)} + \log \abs{x-y} + O(r)
\]
and also,
\[
\log \abs{1 - \overline{\varphi(x)}\varphi(y)} = \log \left( 1 - \abs{\varphi(x)}^2 \right) + O \left( \frac{\abs{y-x}}{1-\abs{\varphi(x)}^2} \right)
=
\log \left( 1 - \abs{\varphi(x)}^2 \right) + O(r).
\]
Hence
\[
2G_D(x,y) - \log (\CR(x,D) \CR(y,D))
=
\log \frac{1}{\abs{x-y}} + O(r).
\]
As $r<R/2$, $x$ and $y$ are in the bulk of $Q(x,R)$ and we can show similarly that
\[
2G_{Q(x,R)}(x,y) - \log (\CR(x,Q(x,R)) \CR(y,Q(x,R)))
=
\log \frac{1}{\abs{x-y}} + O\left( \frac{r}{R} \right).
\]
Coming back to \eqref{eq:proof_harmonic_covariance1}, it shows \eqref{eq:proof_harmonic_covariance}.

\medskip
We now come back to the core of the proof of \eqref{eq:lem delete harmonic}. In the rest of the proof, the constants may depend on $\beta$ and on $R_0$.
Let $\lambda >0$.
By the triangle inequality, the probability on the left hand side of \eqref{eq:lem delete harmonic} is at most
\begin{equation}
\label{eq:proof lem delete harmonic1}
\Prob{ \abs{ \int_{Q(x,r)} \mu^{(R)}(y) \left( e^{i\beta (h^{(R)}(y) - h^{(R)}(x)) } - 1 \right) dy } > \lambda r^{2-\frac{\beta^2}{2}}  }.
\end{equation}
Let $N \geq 1$.
Recall the notation $\Ec'(\Gamma^{(R)};\mathbf{y};\mathbf{z})$ defined in \eqref{E:defEcal'}.
As $\mu^{(R)}$ and $h^{(R)}$ are independent, we have by Lemma \ref{L:moment_JSW},
\begin{align}
& \Expect{ \abs{ \int_{Q(x,r)} \mu^{(R)}(y) \left( e^{i\beta (h^{(R)}(y) - h^{(R)}(x))} - 1 \right) dy }^{2N} } \label{eq:proof lem delete harmonic4}\\
& = \int_{Q(x,r)^N \times Q(x,r)^N} d \mathbf{y} \,d\mathbf{z} \,e^{\beta^2 \Ec'(\Gamma^{(R)};\mathbf{y};\mathbf{z})} 
\Expect{ \prod_{j=1}^N \left( e^{i\beta (h^{(R)}(y_j) - h^{(R)}(x))} - 1 \right) \left( e^{-i\beta (h^{(R)}(z_j) - h^{(R)}(x))} - 1 \right)  }. \nonumber
\end{align}
By a generalised Cauchy--Schwarz inequality, we can bound the expectation on the right hand side of the above display by
\[
\max_{y \in Q(x,r)} \Expect{ \abs{ e^{i\beta (h^{(R)}(y) - h^{(R)}(x))} - 1 }^{2N} }.
\]
Denoting $\Gamma^{Q(0,1)}$ the GFF in the unit square, we deduce by scaling that \eqref{eq:proof lem delete harmonic4} is at most
\begin{align}
& C^N R^{(2-\beta^2/2)2N}  \max_{y \in Q(x,r)} \Expect{ \abs{ e^{i\beta (h^{(R)}(y) - h^{(R)}(x))} - 1 }^{2N} }
\int_{Q(0,\frac{r}{R})^N \times Q(0,\frac{r}{R})^N} d \mathbf{y} \,d\mathbf{z} \,e^{\beta^2 \Ec'(\Gamma^{Q(0,1)};\mathbf{y};\mathbf{z})}. \nonumber
\end{align}
Using \eqref{eq:proof_harmonic_covariance}, we have
\begin{equation}
\label{eq:proof_lem_delete_harmonic_5}
\max_{y \in Q(x,r)} \Expect{ \abs{ e^{i\beta (h^{(R)}(y) - h^{(R)}(x))} - 1 }^{2N} }
\leq C^N \max_{y \in Q(x,r)} \Expect{ \abs{h^{(R)}(y) - h^{(R)}(x)}^{2N} }
\leq (N-1)!! \left( C \frac{r}{R} \right)^N,
\end{equation}
where $(N-1)!!$ is the double factorial of $N-1$, i.e. the product of all positive integers smaller than $N-1$ with the same parity as $N-1$.
Moreover, by Lemmas \ref{L:bound_Ec} and \ref{L:integral} and by scaling, we can bound
\begin{align*}
& \int_{Q(0,\frac{r}{R})^N \times Q(0,\frac{r}{R})^N} d \mathbf{y} \,d\mathbf{z} \,e^{\beta^2 \Ec'(\Gamma^{Q(0,1)};\mathbf{y};\mathbf{z})}
\leq
C^N \int_{Q(0,\frac{r}{R})^{2N}} d \mathbf{y} \prod_{j=1}^{2N} \left( \min_{k \neq j} \abs{y_k - y_j} \right)^{-\beta^2/2} \\
& \qquad =
C^N \left( \frac{r}{R} \right)^{(2-\beta^2/2)2N}
\int_{Q(0,1)^{2N}} d \mathbf{y}  \prod_{j=1}^{2N} \left( \min_{k \neq j} \abs{y_k - y_j} \right)^{-\beta^2/2} \\
& \qquad \leq C^N N^{N \beta^2/2} \left( \frac{r}{R} \right)^{(2-\beta^2/2)2N}.
\end{align*}
Putting things together, it implies that there exists $C>0$ such that the expectation in \eqref{eq:proof lem delete harmonic4} is not larger than
\begin{equation} \label{eq:proof lem delete harmonic5}
(N-1)!! N^{N\beta^2/2} \left( C \frac{r}{R} r^{(2-\beta^2/2)2} \right)^N.
\end{equation}
By Markov's inequality and using that $(N-1)!! \leq C^N N^{N/2}$, we deduce that the probability on the left hand side of \eqref{eq:lem delete harmonic} is at most
\[
\Big( C \frac{r}{R} \frac{1}{\lambda^2} \Big)^N N^{(\beta^2+1)N/2}.
\]
We optimise in $N$ and choose $N = \floor{ e^{-1} (\lambda^2 \frac{R}{r})^{2/(\beta^2+1)} }$ which leads to the estimate on the right hand side of \eqref{eq:lem delete harmonic}.
It concludes the proof of \eqref{eq:lem delete harmonic}.

The proof of \eqref{eq:lem delete harmonic2} about the whole-plane GFF $\Gamma^*$ is simpler since there is no boundary to deal with and follows from the same lines.
\end{proof}

\begin{lemma}\label{lem:tail}
Let $0<2r<R$ and denote by $\mu^{(R)}$ the imaginary chaos associated to the GFF in $Q(0,R)$. There exists $c>0$ which may depend on $\beta$ s.t.
\begin{align}
\label{eq:lem_tail}
& \abs{ \Prob{ \abs{ \mu^{(R)}(Q(0,r)) } > r^{2-\beta^2/2} t } - \exp \left( - (1+o(1)) c^*(\beta) t^{4/\beta^2} \right) }
 \leq \frac{1}{c} \exp \left( - c \left( \frac{R}{r}\right)^{\frac{2}{\beta^2+1}} \right),
\end{align}
where $c^*(\beta)>0$ is the constant from Theorem \ref{T:LSZ} and where $o(1) \to 0$ when $t \to \infty$, uniformly in $r,R$ as above.
\end{lemma}

\begin{proof}[Proof of Lemma \ref{lem:tail}]
Let $\eps>0$.
The probability in \eqref{eq:lem_tail} being equal to
\[
\Prob{ \abs{ \mu^{(R/r)}(Q(0,1))} > t },
\]
we can assume without loss of generality that $r=1$ and $R > 2$.
The tail estimate \eqref{eq:rem tail} concerns the whole-plane GFF $\Gamma^*$ that we write as the independent sum $\Gamma^{(R)} + {h^*}^{(R)}$
where $\Gamma^{(R)}$ is the GFF in $Q(0,R)$ with Dirichlet boundary condition and ${h^*}^{(R)}$ is harmonic in $Q(0,R)$.  By Lemma \ref{lem:delete harmonic}, there exists $c=c(\beta)>0$ s.t.
\begin{align*}
& \Prob{ \abs{ \abs{ \mu^{(R)}(Q(0,1)) } - \abs{\mu^*(Q(0,1)) } } > 1 } \leq \frac{1}{c} \exp \left( - c R^{\frac{2}{\beta^2+1}} \right).
\end{align*}
Hence we have the following upper bound
\[
\Prob{ \abs{ \mu^{(R)}(Q(0,1)) } > t } 
\leq \Prob{ \abs{ \mu^*(Q(0,1)) } > t - 1 } + \frac{1}{c} \exp \left( - c R^{\frac{2}{\beta^2+1}} \right)
\]
and an analogous lower bound. Together with \eqref{eq:rem tail}, this shows \eqref{eq:lem_tail}.
\end{proof}

\section{Hölder regularity and monofractality: proof of Proposition \ref{prop:Holder regularity} and Theorem~\ref{th:monofractality}}\label{S:monofractality}

We start this section by proving Proposition \ref{prop:Holder regularity}.

\begin{proof}[Proof of Proposition \ref{prop:Holder regularity}]
Fix a compact subset $K$ of $D$ and $N \geq 1$. If $A = Q(x,r)$ and $A'=Q(x',r')$ are two squares included in $K$, we have
\begin{align*}
\Expect{ \abs{ \mu(A) - \mu(A') }^{2N} }
& =
\Expect{ \abs{ \mu(A \backslash A \cap A') - \mu(A' \backslash A \cap A') }^{2N} } \\
& \leq 2^{2N-1} \Expect{ \abs{ \mu(A \backslash A \cap A')}^{2N} + \abs{ \mu(A' \backslash A \cap A')}^{2N}}.
\end{align*}
$A \backslash A \cap A'$ and $A' \backslash A \cap A'$ can be covered by finitely many $a \times b$ rectangles with either $a \leq 2(r \vee r')$ and $b \leq C \norme{(x,r)-(x',r')}$, or the bounds swapped for $a$ and $b$.
By Lemma \ref{lemma:rectangle1}, we deduce that
\begin{align*}
\Expect{ \abs{ \mu(A) - \mu(A') }^{2N} }
& \leq C_{K,N} \left\{ (r \vee r') \norme{(x,r) - (x',r')} \right\}^{(2-\beta^2/2) N}.
\end{align*}
With the help of Lemma \ref{lem:Kolmogorov} which is a slight modification of Kolmogorov's continuity theorem, it shows that $(x,r) \mapsto \mu(Q(x,r))$ restricted to the set $\{ (x,r) : Q(x,r) \subset K \}$ possesses a continuous modification which satisfies \eqref{eq:prop_Holder}. It concludes the proof.
\end{proof}

We now turn to the proof of Theorem \ref{th:monofractality}.

\begin{proof}[Proof of Theorem \ref{th:monofractality}]
First, by applying \eqref{eq:prop_Holder} to $(x,r)$ and $(x',r')=(x,0)$, we see that for all compact subset $K$ of $D$, $\alpha < 1-\beta^2/4$, there exists a random constant $C_{K,\alpha}$ finite a.s. such that for all $Q(x,r) \subset K$,
\[ 
|\mu(Q(x,r))| \le C_{K,\alpha} r^{2\alpha}.
\]
The lower bound of \eqref{eq:th_monofractality} then follows. It thus only remains to prove the upper bound of \eqref{eq:th_monofractality}.
Let $\Gamma^*$ be the whole-plane GFF and let $\mu^*$ be the associated imaginary chaos. As with the the GFF $\Gamma$ in the domain $D$, we can make sense of $\mu^*(Q(x,r))$ simultaneously for all $x \in \C$ and $r>0$. Moreover, for all compact sets $K \subset \C$ and $\alpha < 1-\beta^2/4$, there exists a random constant $C_{K,\alpha}$ almost surely finite such that for all $Q(x,r), Q(x',r') \subset K$,
\begin{equation}
\label{eq:proof_th_mono1}
\abs{ \mu^*(Q(x,r)) - \mu^*(Q(x',r')) } \leq C_{K,\alpha} (r \vee r')^\alpha \norme{ (x,r) - (x',r') }^\alpha.
\end{equation}
We are going to prove that almost surely, simultaneously for all $x \in D$,
\begin{equation}
\label{eq:proof_th_mono}
\liminf_{r \to 0} \frac{\log \abs{\mu^*(Q(x,r))}}{ \log r} \leq 2 - \beta^2/2.
\end{equation}
It will be convenient to consider the whole-plane GFF instead of $\Gamma$ because of its translation and scaling invariance.
We now explain how we conclude from \eqref{eq:proof_th_mono} the proof of the upper bound of \eqref{eq:th_monofractality}. By spatial Markov property (\cite[Proposition 2.8]{ImaginaryGeometryIV}), we can couple the two GFFs $\Gamma^*$ and $\Gamma$ such that $\Gamma^* = \Gamma + h$ and such that $\Gamma$ and $h$ are independent and $h$ is harmonic in $D$. We claim that the map
\[
\begin{array}{ccc}
\{ (x,r) \in D \times [0,1], Q(x,r) \subset D \} & \to & \C \\
(x,r) & \mapsto & e^{-i \beta h(x) } \mu^*(Q(x,r)) - \mu(Q(x,r))
\end{array}
\]
possesses a continuous modification. Moreover, for all compact subsets $K$ of $D$ and $\alpha < 3-\beta^2/2$, there exists a random constant $C_{K,\alpha}$ almost surely finite such that for all $Q(x,r) \subset K$,
\[
\abs{ e^{-i \beta h(x) } \mu^*(Q(x,r)) - \mu(Q(x,r)) } \leq C_{K,\alpha} r^{\alpha}.
\]
The proof of this claim is very similar to the proof of Proposition \ref{prop:Holder regularity} with the additional ingredient (which is similar to \eqref{eq:proof_harmonic_covariance}) that for all compact subsets $K$ of $D$, there exists $C_K >0$ such that for all $Q(x,r) \subset K$,
\[
\sup_{y \in Q(x,r)} \Expect{(h(x) - h(y))^2 } \leq C_K r.
\]
We omit the details of the proof of this claim. Since $\alpha$ can be chosen larger than $2-\beta^2/2$, we now see that this claim and \eqref{eq:proof_th_mono} imply the desired upper bound of \eqref{eq:th_monofractality}.

The rest of the proof will consist in proving \eqref{eq:proof_th_mono}. Let us denote by
\[
\xi = 2-\beta^2/2.
\]
Let $\delta >0$ and $K$ be an integer larger than $1/\delta-1$. We are going to bound the probability of the event that there exists $x \in D$ such that
\begin{equation*}
\liminf_{r \to 0} \frac{\log \abs{\mu^*(Q(x,r))}}{\log r} \geq \xi+ 2\delta.
\end{equation*}
For such an $x$, we have for small enough $r>0$,
\begin{equation}
\label{eq:proof_thm_mono}
\abs{\mu^*(Q(x,r))} \leq r^{\xi+\delta}.
\end{equation}
Let $p \geq 1$, $z \in 2^{-2p}\Z^2$ be such that $\abs{z-x} \leq 2^{-2p}$. For all $j = 0, \dots, 2K$, define $r_j := 2^{-p -j}$. By \eqref{eq:proof_th_mono1} with $\alpha > (\xi+\delta)/3$, there exists a random constant $C(\omega)$ a.s. finite such that for all $j = 0, \dots, 2K$,
\begin{align*}
\abs{\mu^*(Q(z,r_j)) - \mu^*(Q(x,r_j))} \leq C(\omega) (r_j 2^{-2p})^\alpha \leq C(\omega) r_j^{3\alpha} \leq r_j^{\xi+\delta}
\end{align*}
if $p$ is large enough. With \eqref{eq:proof_thm_mono}, it implies that if $p$ is large enough, for all $j = 0, \dots, 2K$,
\[
\abs{\mu^*(Q(z,r_j))} \leq 2 r_j^{\xi+\delta}.
\]
To put in a nutshell, we have shown that, up to an event of vanishing probability,
\begin{align}
& \left\{ \exists x \in D, \liminf_{r \to 0} \frac{\log \abs{\mu^*(Q(x,r))}}{\log r} \geq \xi+ 2\delta \right\} \nonumber \\
& \subset \bigcup_{n} \bigcap_{p \geq n} \bigcup_{\substack{z \in 2^{-2p}\Z^2\\z \in D}} \left\{ \forall j = 0, \dots, 2K, \abs{\mu^*(Q(z,r_j))} \leq 2 r_j^{\xi+\delta} \right\} \nonumber \\
& \subset \bigcap_{n} \bigcup_{p \geq n} \bigcup_{\substack{z \in 2^{-2p}\Z^2\\z \in D}} \left\{ \forall j = 0, \dots, 2K, \abs{\mu^*(Q(z,r_j))} \leq 2 r_j^{\xi+\delta} \right\} \label{eq:goal}
\end{align}
the last inclusion being elementary and not specific to the events considered here.
The goal is now to show that the probability of the event \eqref{eq:goal} vanishes. Let $p$ be a large integer and $z \in 2^{-2p} \Z^2 \cap D$. By translation and scaling invariance,
\begin{align*}
\Prob{\forall j = 0, \dots, 2K, \abs{\mu^*(Q(z,r_j))} \leq 2 r_j^{\xi+\delta} }
& = \Prob{ \forall j = 0, \dots, 2K, \abs{\mu^*(Q(0,2^{-j}))} \leq 2 \cdot 2^{-p \delta} 2^{-(\xi + \delta)j} }.
\end{align*}
In Lemma \ref{L:dec23} below the current proof, we show that the above right hand side term is at most $C_K 2^{-2p \delta (K+1)}$ for some $C_K>0$ independent of $p$.
Since $K+1 > 1/\delta$, a union bound concludes that the probability of the event \eqref{eq:goal} vanishes. It concludes the proof.
\end{proof}

\begin{lemma}\label{L:dec23}
Let $\mu^*$ be the imaginary chaos associated to the whole-plane GFF. For all integers $K \geq 1$, there exists $C_K>0$ such that for all $\eps >0$,
\[
\Prob{\forall j=0, \dots, 2K, |\mu^*(Q(0,2^{-j}))| \leq \eps} \leq C_K \eps^{2(K+1)}.
\]
\end{lemma}

Before proving Lemma \ref{L:dec23}, we need to collect a density result following from \cite{aru2022density}.
This result will feature a quantity $c_1(f)$ associated to some complex-valued function $f : D \to \C$ that we define now.
For $\delta >0$, let $\mathcal{Q}_\delta$ be the collection of cubes of the form
\begin{equation}
\label{E:Qdelta}
[4k_1 \delta, (4k_1+1) \delta] \times [4k_2 \delta, (4k_2+1) \delta]
\end{equation}
with $k_1, k_2 \in \Z$.
Consider also the following cones of angle $\pi/4$: 
\begin{equation}
\label{E:cone}
\text{Cone}_k = \{ z \in \C: \arg z \in [k\pi/8,\pi/4+8k\pi/8]\},
\end{equation}
for $k=0, \dots, 15$.
$c_1(f)$ is then defined by
\begin{equation}\label{E:c1f}
c_1(f) := \liminf_{\delta \to 0} \delta^2 \# \{ Q \in \mathcal{Q}_\delta: \exists k \in \{0, \dots, 15\}, \forall x \in Q,  f(x) \in \text{Cone}_k \text{ and } |f(x)| \geq \norme{f}_\infty/2 \}.
\end{equation}
Notice that $c_1(f)>0$ for any nonzero continuous function $f$.

\begin{theorem}\label{T:density}
Let $D$ be a bounded simply connected domain and $\Gamma$ be the GFF in $D$ with Dirichlet boundary conditions. Let $K$ be a compact subset of $D$ and $f: D \to \C$ be a bounded measurable function supported in $K$. Let $\mu(f)$ be the random variable defined by
\[
\mu(f) = \lim_{\eps \to 0} \int_D \eps^{-\beta^2/2} e^{i \beta \Gamma_\eps(x)} f(x) d x,
\]
where $\Gamma_\eps$ is the circle average approximation of $\Gamma$.
Then the law of $\mu(f)$ has a smooth density $\rho$ with respect to Lebesgue measure on $\C$. Moreover, there exists $C>0$ that may depend on $D$, $K$ and $\beta$ but not on $f$ such that $\norme{\rho}_\infty \leq C c_1(f)^{-8} \norme{f}_\infty^{-2}$.
\end{theorem}

\begin{proof}[Proof of Theorem \ref{T:density}]
This is a version of \cite[Theorem 3.6]{aru2022density} that is quantitative in terms of the test function $f$. It follows directly from \cite{aru2022density} by carefully recording the dependence of $f$. We now give a few details. First of all, the imaginary chaos considered in \cite{aru2022density} is defined via a different approximation procedure where the normalisation is done using $e^{\frac{\beta^2}{2} \text{var}(\Gamma_\eps(x))}$ instead of $\eps^{-\beta^2/2}$. This effectively changes the test function $f$ to $\tilde{f} : x \mapsto f(x) \text{CR}(x,D)^{-\beta^2/2}$ where $\text{CR}(x,D)$ stands for the conformal radius of $D$ seen from $x$. Since $c_1(\tilde{f})^{-8} \norme{\tilde{f}}_\infty^{-2}$ differs from $c_1(f)^{-8} \norme{f}_\infty^{-2}$ by at most a multiplicative constant depending only on $D$ and $K$, this will not affect the result.

In \cite[Theorem 3.6]{aru2022density}, the test function $f$ was assumed to be real-valued and continuous.
As we are about to explain, complex-valued test functions can be considered and the continuity assumption can be replaced by the assumption that $c_1(f) >0$. We will give some details about this point while explaining where the upper bound on $\norme{\rho}_\infty$ comes from. By a simple scaling argument, we can assume that $\norme{f}_\infty = 1$. All the bounds obtained in the proof of \cite[Theorem 3.6]{aru2022density} are uniform in $f$ with $\norme{f}_\infty =1$ and work in the exact same way for complex-valued test functions, except the bounds in the proof of \cite[Lemma 6.8]{aru2022density}.
In the proof of this lemma and with the notation therein, one has to bound from below
\begin{align*}
|\scalar{DM, h_\delta}_H|
= \beta \sum_{Q \in \mathcal{Q}_\delta} \int_{Q \times Q} f(x) \overline{f(y)} :e^{i \beta (\hat{Y}_\delta(x) + Z(x)}: :e^{-i \beta (\hat{Y}_\delta(y) + Z(y)}: R_\delta(x,y) dx dy.
\end{align*}
$\mathcal{Q}_\delta$ is the collection of squares defined in \eqref{E:Qdelta}.
Since each term in the above sum is a nonnegative real number, one can bound from below $|\scalar{DM, h_\delta}_H|$ by summing over specific squares $Q \in \mathcal{Q}_\delta$, instead of summing over all squares in $\mathcal{Q}_\delta$. In \cite{aru2022density} and when $f$ is real-valued, we chose to sum only on the squares $Q$ such that either $f(x) \geq \norme{f}_\infty/2$ for all $x \in Q$, or $f(x) \leq -\norme{f}_\infty/2$ for all $x \in Q$. For such a square $Q$, we have $f(x) \overline{f(y)} \geq \norme{f}_\infty^2/4$ for $(x,y) \in Q \times Q$.
In the proof of \cite[Lemma 6.8]{aru2022density}, the total number of squares $Q \in \mathcal{Q}_\delta$ with the above property is bounded from below by some constant $c_1$ times $\delta^{-2}$. $|\scalar{DM, h_\delta}_H|$ is then successfully lower bounded and shown to scale like $c_1$.

If $f$ is complex-valued and recalling the definition \eqref{E:cone} of $\text{Cone}_k$, we can instead sum over squares $Q$ such that there exists $k \in \{0, \dots, 15\}$ so that for all $x \in Q$, $f(x) \in \text{Cone}_k$ and $|f(x)| \geq \norme{f}_\infty/2$. For such a square $Q$, we now have for all $(x, y) \in Q \times Q$, $|f(x) \overline{f(y)}| \geq \norme{f}_\infty^2/4$ and $\arg(f(x) \overline{f(y)}) \in [-\pi/4,\pi/4]$. In particular, the real part of $f(x) \overline{f(y)}$ is at least $\norme{f}_\infty^2 \cos(\pi/4)/4$. The rest of the argument goes through. In particular, recalling the definition \eqref{E:c1f} of $c_1(f)$, if $\delta$ is small enough, we can lower bound the number of appropriate squares by $c_1(f) \delta^{-2}/2$ and give a lower bound on $|\scalar{DM, h_\delta}_H|$ by some quantity scaling like $c_1(f)$.
Together with Lemma 3.4 and Proposition 3.5 of \cite{aru2022density}, this shows that $\norme{\rho}_\infty \leq C c_1(f)^{-8}$.
\end{proof}

\begin{proof}[Proof of Lemma \ref{L:dec23}]
Let $Q_j = Q(0,2^{-j})$, $j=0, \dots, 2K$, and $Q_{2K+1} = \varnothing$.
We start with the elementary observation that
\begin{align*}
& \Prob{ \forall j = 0, \dots, 2K, \abs{\mu^*(Q_j)} \leq \eps }
\leq \Prob{ \forall j = 0, \dots, 2K, \abs{\mu^*(Q_j \backslash Q_{j+1})} \leq 2\eps} \\
& \leq \Prob{ \forall j=0, 2, 4, \dots, 2K, \abs{\mu^*(Q_j \backslash Q_{j+1})} \leq 2\eps}.
\end{align*}
In  the last step, we neglected the constraints on odd values of $j$ so that the sets $Q_j \setminus Q_{j+1}$, $i=0, 2, 4, \dots, 2K$, are now disjoint and well separated: the minimal distance between two such sets is at least $2^{-2K}$.
For $j=0, \dots, K$, let us define $A_j = Q_{2j} \setminus Q_{2j+1}$ and $B_j = Q(0,2^{-2j+1/2}) \setminus Q(0,2^{-2(j+1)+1/2})$. By construction, the $B_j$'s are disjoint and for all $j =0, \dots, K$, $A_j$ is contained in $B_j$ and is at a positive distance to $\partial B_j$. We apply Markov's property inside each $B_j$: there exists a field $h$ such that for all $j = 0, \dots, K$, $\Gamma_j := (\Gamma^* - h) \mathbf{1}_{B_j}$ has the law of a GFF in $B_j$. Moreover, $\Gamma_0, \dots, \Gamma_K$ and $h$ are mutually independent and $h$ is harmonic inside each $B_j$.
Conditionally on $h$, we can now view $\mu^*(A_j)$ as the imaginary chaos $\mu_j$ of $\Gamma_j$ integrated against the test function
\[
f_j(x) := e^{i \beta h(x)} \indic{x \in A_j}.
\]
Thus,
\begin{align*}
\Prob{ \forall j=0, \dots, K, \abs{\mu^*(A_j)} \leq 2\eps}
= \Expect{ \prod_{j=0}^K \Prob{|(\mu_j,f_j)| \leq 2 \eps \vert h} },
\end{align*}
where the expectation integrates out the randomness of $h$. By Theorem \ref{T:density}, for all $j=0, \dots, K$ and conditionally given $h$, the law of the random variable $(\mu_j,f_j)$ is absolutely continuous with respect to Lebesgue measure on $\C$. Moreover, the associated densities are uniformly bounded by some deterministic constant $C$ times $c_1(f_j)^{-8}$. We deduce that for all $j=0, \dots, K$, $\Prob{|(\mu_j,f_j)| \leq 2 \eps \vert h} \leq C c_1(f_j)^{-8} \eps^2$, implying that
\begin{align*}
\Prob{ \forall j=0, \dots, K, \abs{\mu^*(A_j)} \leq 2\eps}
& \leq C \Expect{ \prod_{j=0}^K c_1(f_j)^{-8}} \eps^{2(K+1)} \\
& \leq C \prod_{j=0}^K \Expect{ c_1(f_j)^{-8(K+1)}}^{\frac{1}{K+1}} \eps^{2(K+1)},
\end{align*}
where we used H\"older's inequality (more precisely, a generalisation to finitely many variables) in the last inequality.

To conclude the proof, it remains to show that $\Expect{c_1(f_j)^{-8(K+1)}}$ is finite for all $j$.
Let $j \in \{0,\dots, K\}$ and $x_j \in A_j$ be at distance at least $2^{-10K}$ to the boundary of $A_j$. Let $\eta \in (0,2^{-10K})$. Note that if $|h(x_j) - h(y)| \leq 1/(100\beta)$ for all $y \in B(x_j,\eta)$, then there exists $k \in \{0, \dots, 15\}$ such that $f_j(y) \in \text{Cone}_k$ for all $y \in B(x_j,\eta)$ and thus $c_1(f_j) \geq c \eta^2$ for some constant $c>0$. Thus,
\begin{align*}
\Prob{c_1(f_j) \leq c \eta^2} \leq
\P \Big( \sup_{y \in B(x_j,\eta)} |h(x_j) - h(y)| > 1/(100\beta) \Big).
\end{align*}
Since $h-h(x_j)$ is a smooth Gaussian field inside each $A_j$, then, by local chaining inequalities (e.g. \cite[Proposition 5.35]{van2014probability}),
the above right hand side term is at most $C e^{-c / \eta^2}$. This shows that the tails of $1/c_1(f_j)$ decays exponentially fast. In particular, all negative moments of $c_1(f_j)$ are finite, concluding the proof.
\end{proof}

\section{Law of the iterated logarithm: proof of Theorem \ref{th:LIL}}\label{S:LIL}

In this section, we will write
\[
\varphi(r) \coloneqq r^{2-\beta^2/2} (\log \abs{ \log r})^{\beta^2/4}, \qquad r \in (0,1).
\]
The goal of this section is to establish the law of the iterated logarithm \eqref{eq:th_LIL} in Theorem \ref{th:LIL}. The proof of this result will be decomposed into two main steps. 
First, we will show that \eqref{eq:th_LIL} holds along any geometric sequence of radii:

\begin{lemma}\label{lem:LIL_sequence}
Let $x_0 \in D$ and $\gamma \in (0,1)$ be fixed. Then
\begin{equation}
\label{eq:lem_LIL}
\limsup_{n \to \infty} \frac{\abs{\mu(Q(x_0,\gamma^n))}}{\varphi(\gamma^n)} =  c^*(\beta)^{-\beta^2/4} \quad \quad \mathrm{a.s.}
\end{equation}
\end{lemma}

We will then control the oscillations between two consecutive points in the geometric sequence:

\begin{lemma}\label{lemma:sup_summable} Let $x_0 \in D$.
  For any $\varepsilon > 0$ there exists $\gamma \in (0,1)$ such that
  \[\sum_{n=1}^\infty \P\Big[\sup_{r \in [\gamma^{n+1},\gamma^n]} \abs{ \mu(Q(x_0,\gamma^n)) - \mu(Q(x_0,r))}  > \varepsilon \varphi(\gamma^{n+1})\Big] < \infty.\]
\end{lemma}

Let us first see how the two lemmas imply Theorem~\ref{th:LIL}.

\begin{proof}[Proof of Theorem \ref{th:LIL}] Let $x_0 \in D$.
For any $\gamma \in (0,1)$ fixed, we can lower bound
\[
\limsup_{r \to 0} \frac{\abs{\mu(Q(x_0,r))}}{\varphi(r)}
\geq \limsup_{n \to \infty} \frac{\abs{\mu(Q(x_0,\gamma^n))}}{\varphi(\gamma^n)}.
\]
By Lemma \ref{lem:LIL_sequence}, the right hand side term is almost surely equal to $c^*(\beta)^{-\beta^2/4}$ which provides the desired lower bound in Theorem \ref{th:LIL}.
It remains to prove the upper bound. By Lemma \ref{lem:LIL_sequence}, to prove this upper bound, it is enough to show that for any $\eps >0$, 
there exists $\gamma \in (0,1)$ such that we have
  \begin{equation}\label{eq:difference}
    \limsup_{n \to \infty} \sup_{r \in [\gamma^{n+1},\gamma^n]} \left| \frac{\abs{\mu(Q(x_0,r))}}{\varphi(r)} - \frac{\abs{\mu(Q(x_0,\gamma^n))}}{\varphi(\gamma^n)}\right| \le \varepsilon
    \quad \quad \text{a.s.}
  \end{equation}
  To show \eqref{eq:difference}, we note that, by the triangle inequality and the fact that $\varphi(r)$ is increasing near $0$, we get for $r \in [\gamma^{n+1},\gamma^n]$ and $n$ large enough that
  \begin{align*}
    & \left| \frac{\abs{\mu(Q(x_0,r))}}{\varphi(r)} - \frac{\abs{\mu(Q(x_0,\gamma^n))}}{\varphi(\gamma^n)}\right| \\
    & = \left| \Big( \frac{1}{\varphi(r)} - \frac{1}{\varphi(\gamma^n)} \Big) \abs{\mu(Q(x_0,\gamma^n))} + \frac{\abs{\mu(Q(x_0,r))} - \abs{\mu(Q(x_0,\gamma^n))}}{\varphi(r)} \right| \\
    & \le \Big(\frac{\varphi(\gamma^n)}{\varphi(\gamma^{n+1})} - 1\Big)\frac{\abs{\mu(Q(x_0,\gamma^n))}}{\varphi(\gamma^n)} + \frac{\abs{\mu(Q(x_0,\gamma^n))-\mu(Q(x_0,r))}}{\varphi(\gamma^{n+1})}.
  \end{align*}
  We conclude the proof by noticing that in the first term
\[\lim_{n \to \infty} \frac{\varphi(\gamma^n)}{\varphi(\gamma^{n+1})} - 1 = \gamma^{\frac{\beta^2}{2}-2} \lim_{n \to \infty} \frac{(\log \log \frac{1}{\gamma^{n}})^{\frac{\beta^2}{4}}}{(\log \log \frac{1}{\gamma^{n+1}})^{\frac{\beta^2}{4}}} - 1 = \gamma^{\frac{\beta^2}{2} - 2} - 1,\]
which can be made as small as desired by choosing $\gamma$ close enough to $1$.
  The second term is handled by Lemma~\ref{lemma:sup_summable} and Borel--Cantelli.
\end{proof}

The rest of this section is devoted to the proof of Lemmas \ref{lem:LIL_sequence} and \ref{lemma:sup_summable}.

\subsection{Proof of Lemma \ref{lem:LIL_sequence}}\label{SS:LIL_sequence}

\begin{proof}[Proof of Lemma \ref{lem:LIL_sequence}]
Fix $x_0 \in D$, $\gamma \in (0,1)$ and let $r_n \coloneqq \gamma^n$. We start by proving the lower bound which is the most difficult one. Consider the subsequence $(r_n')$ of $(r_n)$ given by $r_n' = r_{k_n}$ with $k_n = \floor{K n \log \log n }$ for some $K>0$. In particular, if $K$ is large enough, it satisfies
\begin{equation}
\label{eq:proof_prop_LIL0}
\liminf_{n \to \infty} \frac{\log (r'_{n-1}/r'_n)}{\log \log n} > \beta^2+1
\qquad \mathrm{and} \qquad
\lim_{n \to \infty} \frac{\log \abs{\log r'_n}}{\log n} = 1.
\end{equation}
As
\[
\limsup_{n \to \infty} \frac{\abs{ \mu(Q(x_0,r_n)) }}{\varphi(r_n)}
\geq \limsup_{n \to \infty} \frac{\abs{ \mu(Q(x_0,r_n')) }}{\varphi(r_n')},
\]
it is enough to prove the lower bound along the subsequence $(r'_n)$.
In the following, we will decompose the GFF $\Gamma$ as the independent sum of a GFF $\Gamma^{Q(x_0,r'_n)}$ in $Q(x_0,r'_n)$ plus a harmonic function $h^{Q(x_0,r'_n)}$ in $Q(x_0,r'_n)$. We will denote $\mu^{(n)}$ the imaginary chaos associated to $\Gamma^{Q(x_0,r'_n)}$.
Lemma \ref{lem:delete harmonic} together with \eqref{eq:proof_prop_LIL0} implies that
\begin{align*}
\sum_{n} \Prob{ \abs{  \abs{ \mu(Q(x_0,{r'_n})) } - \abs{ \mu^{(n-1)}(Q(x_0,{r'_n})) } } > {r'_n}^{2-\beta^2/2}}
& \leq C \sum_n \exp \left( - c \left( \frac{r_{n-1}'}{r_n'} \right)^{\frac{2}{\beta^2 + 1}} \right) \\
& \leq C \sum_n \exp \left( -c (\log n)^2 \right) < \infty.
\end{align*}
By Borel-Cantelli lemma, it implies that a.s.
\begin{equation}
\label{eq:proof_prop_LIL3}
\limsup_{n \to \infty} \frac{\abs{ \mu(Q(x_0,{r'_n})) }}{\varphi(r_n')}
=
\limsup_{n \to \infty} \frac{\abs{\mu^{(n-1)}(Q(x_0,{r'_n}))}}{\varphi(r_n')}.
\end{equation}
If $m \in \{n,n+1\}$ and $t \in (1/2,2)$, Lemma \ref{lem:tail} implies that
\begin{equation}
\label{eq:proof_prop_LIL4}
\Prob{ \frac{\abs{ \mu^{(n-1)}(Q(x_0,{r'_m})) }}{{r'_m}^{2-\beta^2/2} } > t^{\beta^2/4} c^*(\beta)^{-\beta^2/4} (\log n)^{\beta^2/4}  }
=
n^{-t + o(1)}.
\end{equation}
In particular, Borel-Cantelli lemma implies that a.s.
\[
\limsup_{n \to \infty} \frac{\abs{ \mu(Q(x_0,{r'_n})) }}{\varphi(r_n')}
=
\limsup_{n \to \infty} \frac{\abs{ \mu^{(n-1)}(Q(x_0,{r'_n}) \backslash Q(x_0,{r'_{n+1}})) }}{\varphi(r_n')}.
\]
Now, the random variables appearing on the right hand side are independent when we restrict ourselves to even or odd integers. Moreover, by \eqref{eq:proof_prop_LIL4} and the triangle inequality, they satisfy for $t \in (1/2,2)$,
\[
\Prob{ \frac{\abs{ \mu^{(n-1)}(Q(x_0,{r'_n}) \backslash Q(x_0,{r'_{n+1}})) }}{{r'_n}^{2-\beta^2/2} } > t^{\beta^2/4} c^*(\beta)^{-\beta^2/4} (\log n)^{\beta^2/4}  }
=
n^{-t + o(1)}.
\]
With this and a simple use of the first and second Borel-Cantelli lemmas, we deduce that a.s.
\[
\limsup_{\substack{n \to \infty\\n \mathrm{~even}}} \frac{\abs{ \mu^{(n-1)}(Q(x_0,{r'_n}) \backslash Q(x_0,{r'_{n+1}}))}}{\varphi(r_n')} = c^*(\beta)^{-\beta^2/4},
\]
concluding the proof of the lower bound of \eqref{eq:th_LIL}.

The upper bound is much easier because for this bound we use the version of Borel--Cantelli lemma which does not require any independence. Consider $R_0 >0$ s.t. $Q(x_0,R_0) \subset D$. Since the ratio $r_n/R_0$ decays much faster than $r'_n/r'_{n-1}$, we can repeat the previous arguments to show the analogue of \eqref{eq:proof_prop_LIL3}:
\[
\limsup_{n \to \infty} \frac{\abs{ \mu(Q(x_0,r_n)) }}{\varphi(r_n)}
=
\limsup_{n \to \infty} \frac{\abs{ \mu^{(R_0)}(Q(x_0,r_n)) }}{\varphi(r_n)}.
\]
With Lemma \ref{lem:tail}, we have for $t \in (1/2,2)$,
\[
\Prob{ \frac{\abs{ \mu^{(R_0)}(Q(x_0,r_n)) }}{r_n^{2-\beta^2/2} } > t^{\beta^2/4} c^*(\beta)^{-\beta^2/4} (\log n)^{\beta^2/4}  }
=
n^{-t + o(1)}.
\]
The previous probabilities are thus summable if $t>1$ and Borel-Cantelli lemma concludes the proof of the upper bound of \eqref{eq:th_LIL}.
\end{proof}

\subsection{Proof of Lemma \ref{lemma:sup_summable}}\label{SS:sup_summable}

This section is dedicated to the proof of Lemma \ref{lemma:sup_summable}. Let $\gamma \in (0,1)$ and $n \geq 1$. Define the function
\begin{equation}\label{E:Fn}
F_n : r \in I_n \mapsto \mu(Q(x_0,\gamma^n) \setminus Q(x_0,r)),
\qquad \text{where} \qquad I_n := [\gamma^{n+1}, \gamma^n].
\end{equation}
In Lemma \ref{lemma:sup_summable}, we want to control the uniform norm of $F_n$. This norm is difficult to study directly. For instance, we do not know \textit{a priori} how to compute its expectation. Our proof will be based on Sobolev embeddings which allow us to control this norm via more tractable norms. Let us first recall some definitions and some results on fractional Sobolev spaces. For more background on Sobolev spaces, we refer to \cite{di2012hitchhikers}.

\medskip

\textbf{Sobolev spaces.}
Let $I \subset \R$ be an interval.
For $s \in (0,1)$ and $p\geq 1$, the Sobolev space $W^{s,p}(I)$ is defined as
\begin{equation}
W^{s,p}(I) := \left\{ f \in L^p(I): \norme{f}_{W^{s,p}(I)}^p := \norme{f}_{L^p(I)}^p + \int_{I \times I} \frac{|f(x) - f(y)|^p}{|x-y|^{sp + 1}} dx dy < \infty  \right\}.
\end{equation}
Recall also that for $\alpha \in (0,1)$, the space of $\alpha$-H\"older continuous functions is defined as
\begin{equation}
C^\alpha(I) := \left\{ f \in C^0(I): \norme{f}_{C^\alpha(I)} := \sup_{x \in I} |f(x)| + \sup_{x\neq y \in I} \frac{|f(x)-f(y)|}{|x-y|^\alpha} < \infty  \right\}.
\end{equation}
The Sobolev embedding we will rely on is the following (see e.g. \cite[Theorem 8.2]{di2012hitchhikers}): 
when $sp > 1$, then $W^{s,p}(I) \subset C^\alpha(I)$ for $\alpha = s-1/p$. Moreover, there exists $C(s)>0$ that may depend on $s$ but not on $p$, nor on $I$, such that for all intervals $I$ and for all $f \in W^{s,p}(I)$,
\begin{equation}
\label{E:Sobolev_embed}
\sup_{x\neq y \in I} \frac{|f(x)-f(y)|}{|x-y|^\alpha} \leq C(s) \left( \int_{I \times I} \frac{|f(x) - f(y)|^p}{|x-y|^{sp + 1}} dx dy \right)^{1/p}
\quad \quad (\alpha=s-1/p).
\end{equation}
The fact that the constant $C(s)$ can be chosen independently of $p$ can be checked by following the proof of \cite[Theorem 8.2]{di2012hitchhikers}.
The fact that the constant can be chosen independently of the interval $I$ follows directly by a scaling argument.

\medskip

With the Sobolev inequality \eqref{E:Sobolev_embed} at hand, the proof of Lemma \ref{lemma:sup_summable} will be based on the following estimate:

\begin{lemma}\label{L:bound_Sobolev_Fn}
Let $F_n$ be the function defined in \eqref{E:Fn}. For all $s \in (0,1-\beta^2/4)$, there exists $C>0$ that may depend on $\beta$ and $s$ such that for all $n \geq 1$ and even integers $p \geq 2$,
\begin{equation}
\label{E:L_Sobolev_Fn}
\Expect{\int_{I_n \times I_n} \frac{|F_n(r)-F_n(r')|^p}{|r-r'|^{sp+1}} dr dr'} \leq C^p p^{\frac{\beta^2}{4}p} (1-\gamma)^{(1-\beta^2/4 -s)p}
\gamma^{(2-\beta^2/2)np} \gamma^{-spn +n}.
\end{equation}
\end{lemma}

\begin{proof}[Proof of Lemma \ref{L:bound_Sobolev_Fn}]
Let $s \in (0,1-\beta^2/4)$, $n \geq 1$ and $p \geq 2$ be an even integer.
By Lemma \ref{lemma:rectangle1}, we have for all $r,r'\in I_n$,
\begin{align*}
\Expect{|F_n(r)-F_n(r')|^p} = \Expect{ \mu(Q(x_0,r\vee r')\setminus Q(x_0,r\wedge r'))}
\leq C^p ((r\vee r')|r-r'|)^{(1-\beta^2/4)p} p^{\frac{\beta^2}{4}p}.
\end{align*}
By Fubini, we deduce that the left hand side of \eqref{E:L_Sobolev_Fn} is at most
\begin{align*}
C^p p^{\frac{\beta^2}{4}p} \int_{I_n \times I_n} (r \vee r')^{(1-\beta^2/4)p} |r-r'|^{(1-\beta^2/4-s)p-1} dr dr'.
\end{align*}
Recalling that $I_n = [\gamma^{n+1},\gamma^n]$, we bound $(r \vee r')^{(1-\beta^2/4)p}$ in the last integral above by $\gamma^{(1-\beta^2/4)np}$. The integral over $I_n \times I_n$ of $|r-r'|^{(1-\beta^2/4-s)p-1}$ is then equal to
\[
\frac{2}{(1-\beta^2/4-s)p((1-\beta^2/4-s)p+1)} (\gamma^n-\gamma^{n+1})^{(1-\beta^2/4-s)p+1}.
\]
This concludes the proof.
\end{proof}

\begin{proof}[Proof of Lemma \ref{lemma:sup_summable}]
Let $\eps>0$ be fixed. Let $\gamma>0$ be close to 1 and $n \geq 1$. Let $s \in (0,1-\beta^2/4)$ be a fixed parameter ($s=1/2$ would do) and $p \geq 2$ be a large even integer that we will choose precisely later. Let $\alpha = s-1/p$. 
We have,
\begin{align*}
\Prob{\sup_{r \in I_n} |F_n(r)| \geq \eps \varphi(\gamma^{n+1}) }
& =\Prob{\sup_{r \in I_n} |F_n(r) - F_n(\gamma^n)| \geq \eps \varphi(\gamma^{n+1}) } \\
& \leq \Prob{\sup_{r\neq r' \in I_n} \frac{|F_n(r)-F_n(r')|}{|r-r'|^\alpha} \geq \eps \varphi(\gamma^{n+1}) (\gamma^n - \gamma^{n+1})^{-\alpha} }.
\end{align*}
By Markov's inequality, this is at most
\[
(\eps^{-1} (1-\gamma)^\alpha)^p \varphi(\gamma^{n+1})^{-p} \gamma^{\alpha p n} \Expect{\left( \sup_{r\neq r' \in I_n} \frac{|F_n(r)-F_n(r')|}{|r-r'|^\alpha} \right)^p}.
\]
Combining the Sobolev inequality \eqref{E:Sobolev_embed} and Lemma \ref{L:bound_Sobolev_Fn}, we can bound the above expectation by
\[
C^p p^{\frac{\beta^2}{4}p} (1-\gamma)^{(1-\beta^2/4 -s)p}
\gamma^{(2-\beta^2/2)np} \gamma^{-spn +n}.
\]
Putting things together and recalling that $\varphi(\gamma^{n+1}) = \gamma^{(2-\beta^2/2)np} \log((n+1) |\log \gamma| )^{\beta^2/4}$, we obtain that
\[
\Prob{\sup_{r \in I_n} |F_n(r)| \geq \eps \varphi(\gamma^{n+1}) }
\leq (1-\gamma)^{-1} (C \eps^{-1} (1-\gamma)^{1-\beta^2/4} (\log n)^{-\beta^2/4})^p p^{\frac{\beta^2}{4}p}.
\]
We can now optimise in $p$. We choose $p = e^{-1} (C \eps^{-1} (1-\gamma)^{1-\beta^2/4} (\log n )^{-\beta^2/4})^{-\frac{4}{\beta^2}}$ (or rather an even integer close to this value), leading to the following bound:
\[
\Prob{\sup_{r \in I_n} |F_n(r)| \geq \eps \varphi(\gamma^{n+1}) }
\leq (1-\gamma)^{-1} \exp \left( - \frac{\beta^2}{4} e^{-1} (C \eps^{-1} (1-\gamma)^{1-\beta^2/4})^{-\frac{4}{\beta^2}} \log n \right).
\]
If $\gamma$ is close enough to $1$, then the coefficient in front of $\log n$ is larger than 1 making the right hand side term of the above display summable in $n$. This concludes the proof.
\end{proof}

\section{Exceptional points: proof of Theorem \ref{th:exceptional points}}\label{S:exceptional}

\begin{proof}[Proof of Theorem \ref{th:exceptional points}]
We can restrict ourselves to the bulk of $D$ by considering squares which are at distance at least $\eta>0$ from the boundary and then let $\eta \to 0$. We start by proving the upper bound of \eqref{eq:th_exceptional_sup} and \eqref{eq:th_exceptional_number}.
Let $a>0$, $d  > 2 - c^*(\beta)a^{4/\beta^2}$ and $\delta,\eps>0$ be small enough so that
\begin{equation}\label{eq:proof_prop_exception1}
\delta(1-\beta^2/4-\eps) - 2\eps > \eps
\mathrm{~and~}
-2 - 2\delta + c^*(\beta) a^{4/\beta^2} + (1+\delta)d > \delta.
\end{equation}
In this part, the constants may depend on $\delta, \eps$ and $\eta$.
With the help of the Hölder estimate \eqref{eq:prop_Holder}, we are going to discretise the space $D$ and the set of radii. We will consider the sequence $r_n = n^{-1/\delta}$.
Let $x \in D$ and $r>0$. Let $n \geq 1$ such that $r_{n+1} \leq r \leq r_n$ and notice that $r_n - r_{n+1} \leq C r_{n+1}^{1+\delta}$. Take $y \in r_{n+1}^{1+\delta} \Z^2 \cap D$ such that $\abs{x-y} \leq r_{n+1}^{1+\delta}$. If
$
\abs{\mu(Q(x,r))} \geq a r^{2-\beta^2/2} \abs{\log r}^{\beta^2/4}
$, then by the Hölder estimate \eqref{eq:prop_Holder} applied to $\alpha = 1 -\beta^2/4-\eps$, we have
\begin{align*}
\abs{\mu(Q(y,r_{n+1}))} & \geq a r^{2-\beta^2/2} \abs{\log r}^{\beta^2/4} - C (r \max(r-r_{n+1}, \abs{y-x}))^{1-\beta^2/4 - \eps} \\
& \geq a r_{n+1}^{2-\beta^2/2} \abs{\log r_{n+1}}^{\beta^2/4} - C r_{n+1}^{2-\beta^2/2-2\eps + \delta(1-\beta^2/4-\eps)}\\
& \geq a r_{n+1}^{2-\beta^2/2} \abs{\log r_{n+1}}^{\beta^2/4} - C r_{n+1}^{2-\beta^2/2+\eps},
\end{align*}
where we used \eqref{eq:proof_prop_exception1} in the last inequality.
In particular, if we denote
\[
\Tc_n(a) \coloneqq \left\{ y \in r_n^{1+\delta} \Z^2 \cap D: \abs{\mu(Q(y,r_n))} \geq a r_n^{2-\beta^2/2} \abs{\log r_n}^{\beta^2/4} - C r_n^{2-\beta^2/2+\eps} \right\},
\]
it implies that
\begin{equation}
\label{eq:proof_prop_exception3}
\Tc(a) \subset \bigcap_{p \geq 1} \bigcup_{n \geq p} \bigcup_{y \in \Tc_n(a)} B(y,r_n^{1+\delta}).
\end{equation}
To conclude that $\dim(\Tc(a)) \leq d$ a.s., it is thus enough to show that
\begin{equation}
\label{eq:proof_prop_exception2}
\sum_{n \geq 1} \# \Tc_n(a) (r_n^{1+\delta})^d
\end{equation}
is a.s. finite.
Lemmas \ref{lem:delete harmonic} and \ref{lem:tail} that we apply with $R_0 = \eta/10$ imply that if $d(x,\partial D) > \eta$,
\begin{equation*}
\Prob{ \abs{ \mu(Q(x,r_n)) } \geq a r_n^{2-\beta^2/2} \abs{ \log r_n}^{\beta^2/4} } = 
\exp \left( - (1+o(1)) c^*(\beta) a^{4/\beta^2} \abs{\log r_n} \right)
\end{equation*}
where $o(1) \to 0$ as $n\to \infty$. This leads to
\[
\Expect{ \# \Tc_n(a) } \leq r_n^{-2-2\delta+c^*(\beta)a^{4/\beta^2}+o(1)}.
\]
Hence, recalling that $r_n = n^{-1/\delta}$, the expected value of \eqref{eq:proof_prop_exception2} is at most
\[
\sum_{n \geq 1} n^{-(-2-2\delta+c^*(\beta)a^{4/\beta^2} + (1+\delta)d)/\delta+o(1)}
\]
which is finite thanks to \eqref{eq:proof_prop_exception1}. In particular it shows that \eqref{eq:proof_prop_exception2} is finite a.s. which concludes the proof of the upper bound of \eqref{eq:th_exceptional_number}. The upper bound of \eqref{eq:th_exceptional_sup} follows as well: if $a > \left( 2/c^*(\beta) \right)^{\beta^2/4} $, we have shown that
\[
\sum_{n \geq 1} \Prob{ \# \Tc_n(a) \geq 1} \leq \sum_{n \geq 1} \Expect{ \# \Tc_n(a) }
\]
is finite. Borel-Cantelli lemma therefore implies that a.s. if $n$ is large enough $\Tc_n(a)$ is empty. \eqref{eq:proof_prop_exception3} then concludes that $\Tc(a)$ is almost surely empty. Since this is true for all $a > \left( 2/c^*(\beta) \right)^{\beta^2/4} $, it shows the upper bound of \eqref{eq:th_exceptional_sup}.

\medskip

Before we move on to the proof of the lower bound, let us recall the notion of limsup fractal (see \cite[Section 10.1]{morters_peres_2010}). Suppose $D=(0,1)^2$, let $\mathfrak{Q}_n\coloneqq\{(k2^{-n},(k+1)2^{-n})\times(l2^{-n},(l+1)2^{-n}),\,0\leq k,l\leq 2^n-1\}$ be the collection of dyadic squares of sidelength $2^{-n}$ for each $n\in\N$, and $\mathfrak{Q}\coloneqq\cup_{n\in\N}\mathfrak{Q}_n$. Let $Z=(Z(Q))_{Q\in\mathfrak{Q}}$ be a collection of Bernoulli random variables indexed by $\mathfrak{Q}$ such that $p_n\coloneqq\prob(Z(Q)=1)$ is independent of $Q\in\mathfrak{Q}_n$. The limsup fractal associated to the process $Z$ is by definition
\[A\coloneqq\bigcap_{n\in\N}\bigcup_{k\geq n}\bigcup_{Q\in\mathfrak{Q}_k:\,Z(Q)=1}Q.\]
For each $Q\in\mathfrak{Q}_m$ and $n\geq m$, we also write
\[M_n(Q)\coloneqq\sum_{Q'\in\mathfrak{Q}_n,\,Q'\subset Q}Z(Q').\]
\cite{morters_peres_2010} provides a lower bound for the Hausdorff dimension of a limsup fractal:

\begin{proposition}\cite[Proposition 10.6]{morters_peres_2010}\label{prop:limsup_fractal}
 Suppose that there exists $\zeta(n)\geq1$ and $\gamma\in(0,2)$ such that the following holds for each $m\in\N$. 
 \begin{enumerate}
 \item For each $Q\in\mathfrak{Q}_m$, and $n\geq m$ we have 
 \[\mathrm{Var}(M_n(Q))\leq\zeta(n)\E[M_n(Q)].\]
 \item $\underset{n\to\infty}{\lim}\zeta(n)\frac{2^{n(\gamma-2)}}{p_n}=0$.
 \end{enumerate}
 Then almost surely $\dim_{\mc{H}}A\geq\gamma$.
\end{proposition}

For $n\in\N$, let $r_n\coloneqq2^{-n}$ and $H_n\coloneqq r_n\Z^2\cap D'$ where $D' \Subset D$. Fix $\delta\in(0,\frac{1}{2})$ and set:
\[
\mc{C}_n(a)\coloneqq\left\lbrace x\in H_n:\,\abs{\mu^{(Q(x,r_n^{1-2\delta}))}(Q(x,r_n^{1-\delta}))} \geq a r_n^{(1-\delta)(2-\frac{\beta^2}{2})} \abs{\log r_n^{1-\delta}}^\frac{\beta^2}{4} \right\rbrace.
\]
The limsup fractal we will be interested in is defined by:
\[A(a)\coloneqq\bigcap_{n\in\N}\bigcup_{k\geq n}\bigcup_{x\in\mc{C}_k(a)}Q(x,r_k).\]
Let us first show that $A(a)\subset\mc{T}(a)$.
To this end, let us start by noticing that Lemma \ref{lem:delete harmonic} and Borel--Cantelli lemma imply that, almost surely, if $n \in \N$ is large enough, then for all $x \in H_n$,
\begin{equation}\label{E:LA1}
    \abs{ \big|\mu(Q(x,r_n^{1-\delta}))\big| - \big|\mu^{(Q(x,r_n^{1-2\delta}))}(Q(x,r_n^{1-\delta}))\big|} \le r_n^{(1-\delta)(2-\frac{\beta^2}{2})}.
\end{equation}
Let $x\in A(a)$. By definition, there exist an increasing sequence of integers $(k_n)_{n \geq 1}$ and a sequence of points $(x_{k_n})_{n \geq 1}$ such that $x_{k_n} \in \mc{C}_{k_n}(a)$ and $x \in Q(x_{k_n},r_{k_n})$ for all $n\in\N$. Now by Proposition~\ref{prop:Holder regularity} applied to any $\alpha$ satisfying $\frac{1-\delta}{1-\frac{\delta}{2}}(1-\frac{\beta^2}{4})<\alpha<1-\frac{\beta^2}{4}$, we find that 
\[|\mu(Q(x,r_{k_n}^{1-\delta}))-\mu(Q(x_{k_n},r_{k_n}^{1-\delta}))|=O(r_{k_n}^{\alpha(1-\delta)}\times r_{k_n}^\alpha)=o\left(r_{k_n}^{(1-\delta)(2-\beta^2/2)}\abs{\log r_{k_n}^{1-\delta}}^{4/\beta^2}\right).\]
Together with \eqref{E:LA1}, we get that
\[
\underset{k\to\infty}{\limsup}\frac{|\mu(Q(x,r_{k_n}^{1-\delta}))|}{r_{k_n}^{(1-\delta)(2-\beta^2/2)}\abs{\log r_{k_n}^{1-\delta}}^{4/\beta^2}}\geq a,
\]
showing that $x\in\mc{T}(a)$.

To obtain a lower bound on the Hausdorff dimension of $\mc{T}(a)$, it is sufficient to obtain one for $A(a)$. To this end, we will show that we can apply Proposition \ref{prop:limsup_fractal}.
Let $M_n(Q)\coloneqq\#(\mc{C}_n(a)\cap Q)$ as above. Notice first that, by Lemma \ref{lem:tail}, $p_n\coloneqq \prob(x\in\mc{C}_n(a))$ is independent of $x$ and equal to  
$r_n^{(1-\delta)c^*(\beta)a^{4/\beta^2}+o(1)}$ as $n\to\infty$. 
Second, we bound the second moment of $M_n(Q)$. Assume $Q=[0,1]^2$ for simplicity. 
Notice that the events $\{x\in \mc{C}_n(a)\}$ and $\{y\in\mc{C}_n(a)\}$ are independent as soon as $x,y\in H_n$ satisfy $|x-y|>2 \sqrt{2} r_n^{1-2\delta}$. 
Therefore,
\begin{align*}
\E\left[M_n(Q)^2\right]
&=\sum_{\substack{x,y\in H_n\\|x-y|>2\sqrt{2}r_n^{1-2\delta}}}\prob\left(x\in\mc{C}_n(a),y\in\mc{C}_n(a)\right)+\sum_{\substack{x,y\in H_n\\|x-y| \leq 2\sqrt{2}r_n^{1-2\delta}}}\prob\left(x\in\mc{C}_n(a),y\in\mc{C}_n(a)\right)\\
&\leq \sum_{\substack{x,y\in H_n\\|x-y|>2\sqrt{2}r_n^{1-2\delta}}}\prob\left(x\in\mc{C}_n(a)\right)\prob\left(y\in\mc{C}_n(a)\right)+\sum_{\substack{x,y\in H_n\\|x-y| \leq 2\sqrt{2}r_n^{1-2\delta}}}\prob\left(x\in\mc{C}_n(a)\right)\\
&\leq\E[M_n(Q)]^2+C r_n^{-4\delta} \E[M_n(Q)].
\end{align*}
Thus, with $\zeta(n)\coloneqq Cr_n^{-4\delta}$, we have:
\[\mathrm{Var}(M_n(Q))\leq\zeta(n)\E[M_n(Q)].\]
Recalling that $p_n = r_n^{(1-\delta)c^*(\beta)a^{4/\beta^2} + o(1)}$, we conclude by Proposition~\ref{prop:limsup_fractal} that a.s. $\dim_{\mc{H}}(A(a))\geq 2-(1-\delta)c^*(\beta)a^{4/\beta^2}-4\delta$. The lower bound of \eqref{eq:th_exceptional_number} follows since $\delta$ can be chosen arbitrarily small. The lower bound of \eqref{eq:th_exceptional_sup} is then a consequence of \eqref{eq:th_exceptional_number} which finishes the proof of Theorem \ref{th:exceptional points}.
\end{proof}

\section{Besov regularity of imaginary chaos: proof of Theorem \ref{thm:regularity}}\label{sec:Besov}

In this section, we will prove Theorem \ref{thm:regularity} on the optimal Besov regularity of the imaginary chaos. 
The proof is based on wavelet analysis and similar ideas were also used in \cite{JSV}. As a new ingredient in the boundary case where the smoothness index $s = -\beta^2/2$, we employ a law of large numbers type argument to show that the quantities $A_j$ (see below) appearing in the wavelet version of the Besov norm converge almost surely to their nonzero mean.
On the other hand, the proof in the case $p=+\infty$ will rely on what we have developed in the earlier sections of the current paper, especially Theorem \ref{th:LIL}.
Let us first recall from \cite[Chapter 6]{Meyer} basic properties of wavelet bases and the characterisation of Besov spaces using them.

\medskip

\textbf{Wavelet bases and Besov spaces.}
The wavelet bases we will use in 2D are defined in terms of a father wavelet $\phi \in C^2(\reals^2)$ and three mother wavelets $\psi^{(\tau)} \in C^2(\reals^2)$ for $\tau \in \{0,1\}^2 \setminus \{(0,0)\}$. For $j \geq 0$, let
\[\Lambda_j := \{ \lambda = 2^{-j} k + 2^{-j-1} \tau : k \in \Z^2, \tau \in \{0,1\}^2 \setminus \{(0,0)\}\}\]
and for $\lambda = 2^{-j} k + 2^{-j-1} \tau \in \Lambda_j$, let $\psi_\lambda(x) = 2^{j} \psi^{(\tau)}(2^j x - k)$.
The father and mother wavelets are such that $\{\phi(\cdot - k), k \in \Z^2\} \cup \bigcup_{j \geq 0} \{ \psi_\lambda, \lambda \in \Lambda_j \}$
forms an orthonormal basis of $L^2(\reals^2)$. 

The usefulness of wavelet bases stems from the fact that the wavelets can be chosen to have several desirable properties at the same time. In particular we will assume that $\supp \phi, \supp \psi^\tau \subset K$ for some square $K \subset \reals^2$.

The localisation and smoothness properties allows one to characterise various function spaces in terms of the wavelet coefficients. The characterisation for Besov spaces that we will use is the following: If $f \in \S'(\reals^2)$ is a distribution with the wavelet series
\begin{equation}\label{E:wavelet_decomposition}
f(x) = \sum_{k \in \integers} \gamma(k) \phi(x - k) + \sum_{j=0}^\infty \sum_{\lambda \in \Lambda_j} \alpha(\lambda) \psi_\lambda(x),
\end{equation}
then a norm equivalent to the Besov norm $\|\cdot\|_{B_{p,q}^s(\reals^2)}$, assuming $|s| < 2$,\footnote{Here $2$ comes from the fact that we chose our wavelets to be $C^2$. More generally one could choose them to be $C^r$ and expand the valid range of $s$ to $|s| < r$.} is given by
\[\|f\|_{B_{p,q}^s(\reals^2)} \sim \|(\gamma(k))_{k \in \integers}\|_{\ell^p} + \Big\|\Big\{2^{j(1 - 2/p + s)}\Big(\sum_{\lambda \in \Lambda_j} |\alpha(\lambda)|^p\Big)^{1/p}\Big\}_{j \ge 1}\Big\|_{\ell^q}.\]
If $D$ is an open subset of $\R^2$, the space $B_{p,q,\loc}^s(D)$ is defined as the space of distributions $f \in \S'(\R^2)$ such that for all $u \in C^\infty_c(D)$, $f u \in B_{p,q}^s(\R^2)$.

\medskip

The second tool we will use is a decomposition of the GFF.

\medskip

\textbf{Decomposition using exactly scaling field.} In \cite[Section~2]{JSV} it was noted that to study the Besov-regularity of complex chaos it is enough to consider the so-called exactly scaling case. Let us briefly recall the details.

Let $X$ be the log-correlated field on $\reals^2$ with the covariance
\begin{equation}
    \label{E:LA2}
    \E[X(x)X(y)] = \log^+ \frac{R}{|x-y|}
\end{equation}
for some $R > 0$, where $\log^+ t = \log(\max(t,1))$. This is indeed a covariance function in two dimensions, since one can write
\[\log^+ \frac{R}{|x-y|} = \int_0^\infty \Big(1 - \sqrt{\frac{e^u |x-y|}{R}}\Big)_+  \, du + 2\Big(1 - \sqrt{\frac{|x-y|}{R}}\Big)_+,\]
which expresses the covariance kernel $\log^+ \frac{1}{|x-y|}$ as a sum (integral) of positive-definite kernels with covariance $(1 - \sqrt{e^u |x-y|})_+$, see \cite{RobertVargas2010} for more details and references. The integral part corresponds to a so-called $\star$-scale invariant field. In particular one can find approximations $X_s(x)$, $s \in (0,R]$, of the field by cutting the integral at level $\log \frac{R}{s}$, leading to a process with the covariance structure
\begin{align}
    \E[X_s(x)X_t(y)] & = \int_0^{\log \frac{R}{s \vee t}} \Big(1 - \sqrt{\frac{e^u|x-y|}{R}}\Big)_+ \, du + 2\Big(1 - \sqrt{\frac{|x-y|}{R}}\Big)_+ \nonumber \\
    & = \log^+ \frac{R}{s \vee t \vee |x-y|} + 2 - 2\sqrt{\frac{|x-y|}{s \vee t \vee |x-y|}}. \label{eq:exactly_scaling_star_scale}
\end{align}

By (the proof of) \cite[Theorem~A]{JSW2} we can for any $K \coloneqq \overline{U} \subset D$ where $U \subset D$ is open find a coupling such that $\Gamma|_K = X + G$, where $G$ is an a.s. smooth Gaussian field. Note that for $\varepsilon < R/2$ we have that $\E[X^{(\varepsilon)}(x)^2] = \log \frac{1}{\varepsilon}$, where $X^{(\varepsilon)}$ is the $\eps$-circle average of $X$. Hence when constructing the chaos for the field $X$ there is no difference in Wick normalisation and renormalisation by $\varepsilon^{-\beta^2/2}$. It follows that if $\nu$ is the chaos constructed using $X$, we have that for any $u \in C_c^\infty(D)$, $u \mu = e^{i\beta G} u \nu$ almost surely. The exponent in front is smooth and hence a.s.
\begin{equation}\label{E:Besov_mu_nu}
\mu \in B^s_{p,q,\loc}(D) \qquad \text{if and only if} \qquad \nu \in B^s_{p,q,\loc}(D).
\end{equation}

The usefulness of the field $X$ (whose covariance is given by \eqref{E:LA2}) stems from the fact that it satisfies the following scaling relation in law: For any $x_0 \in \reals^2$ and $\varepsilon \le 1$, we have $(X(\varepsilon x))_{x \in B(x_0,R/2)} \sim (X(x) + N)_{x \in B(x_0,R/2)}$, where $N$ is a centered Gaussian with variance $\log(1/\varepsilon)$ independent of $X$. The restriction to a ball of radius $R/2$ ensures that the covariance is a pure logarithm and does not get truncated. This implies in particular that if $U \subset B(0,R/2)$, then
\[\int_{\varepsilon U} \nu(x) f(x) \, dx \sim \varepsilon^{2 - \frac{\beta^2}{2}} e^{i\beta N(0,\log \varepsilon^{-1})} \int_U \nu(x) f(\varepsilon x) \, dx.\]
In what follows we will choose $R$ to be so large that the supporting square $K$ of the mother wavelets lies inside $B(0,R/2)$, which lets us employ the scaling relation when computing $\E[|\alpha(\lambda)|^p]$.

\medskip

We are now ready to prove Theorem \ref{thm:regularity}.

\begin{proof}[Proof of Theorem \ref{thm:regularity}]

Recall that we are interested in whether $\mu$, or equivalently $\nu$ (see \eqref{E:Besov_mu_nu}), belongs to $B_{p,q,\loc}^s(D)$ where $s$ equals the critical value $s=-\beta^2/2$. Let $u \in C^\infty_c(D)$ and denote by $\alpha(\lambda)$ and $\gamma(k)$ the coefficients appearing in the wavelet decomposition \eqref{E:wavelet_decomposition} of $u \nu$ where $\nu$ is the exactly scaling chaos defined above. We will denote by
\begin{equation}
    \label{E:def_Aj}
    A_j := 2^{j(1 - 2/p -\beta^2/2)}\Big(\sum_{\lambda \in \Lambda_j} |\alpha(\lambda)|^p\Big)^{1/p}.
\end{equation}
The proof will deal with three cases separately.
\medskip

\textsc{Case 1}: $p<\infty,q=\infty$.
We want to show that $\mu\in B^{-\beta^2/2}_{p,\infty,\loc}(D)$ a.s. 
Since the function $u \in C^\infty_c(D)$ we fixed at the beginning of the proof is arbitrary, it is enough to show that $u \nu \in B^{-\beta^2/2}_{p,\infty}(\R^2)$ a.s.
Note that, by compactness of the support of $u\nu$, the sequence $(\gamma(k))_{k \in \Z}$ appearing in the wavelet decomposition \eqref{E:wavelet_decomposition} of $u\nu$ contains only finitely many nonzero terms. Its $\ell^p$-norm is therefore finite and the goal now is to show that the sequence $(A_j)_{j \geq 1}$, defined in \eqref{E:def_Aj}, is almost surely bounded.
Note also that, again by compactness of the support of $u\nu$, for all $j \geq 1$, the number of nonzero $\alpha(\lambda)$ for $\lambda \in \Lambda_j$ is at most $C 2^{2j}$ which implies by H\"older's inequality that, for all $p'>p$:
\begin{equation}\label{E:pf_holder}
2^{-2p/j} \Big(\sum_{\lambda \in \Lambda_j} |\alpha(\lambda)|^p\Big)^{1/p} \leq C 2^{-2p'/j} \Big(\sum_{\lambda \in \Lambda_j} |\alpha(\lambda)|^{p'}\Big)^{1/p'}.
\end{equation}
This shows that it is actually enough to show that the sequence $(A_j)_{j \geq 1}$ is a.s. bounded when $p \ge 2$ is an even integer.
We will proceed in two steps by showing that
\begin{align}
\label{E:pf_Besov3}
    & (\E[A_j^p])_{j \geq 1} \text{ is bounded,}\\
    \text{and} \qquad & A_j^p - \E[A_j^p] \to 0 \qquad \text{a.s.}
\label{E:pf_Besov4}
\end{align}

Let us start with the proof of \eqref{E:pf_Besov3}.
By Fubini, we have for all $j \geq 1$,
\begin{align*}
\E[A_j^p] = 2^{j(1-2/p-\beta^2/2)p} \sum_{\lambda \in \Lambda_j} \E[|\alpha(\lambda)|^p]
\end{align*}
and for all $\lambda = 2^{-j}k + 2^{-j-1} \tau \in \Lambda_j$,
\begin{align*}
\E[|\alpha(\lambda)|^p] = 2^{jp} \E \abs{ \int \nu(x) u(x) \psi^\tau(2^jx -k) dx}^p.
\end{align*}
Using the scaling property and translation invariance of $\nu$ we see that
\begin{align}
  \E[|\alpha(\lambda)|^p] & = 2^{jp} \E \abs{ \int \nu(x) u(x + 2^{-j}k) \psi^\tau(2^j x) \, dx}^p \nonumber \\
  & = 2^{jp} 2^{p(-2j+j\beta^2/2)} \E \abs{ \int \nu(x) u(2^{-j} (x + k)) \psi^\tau(x) \, dx}^p \label{eq:alpha_scaling} \\
  & \le 2^{-j(1-\beta^2/2)p} \sqrt{\E |\nu(K)|^{2p}} \|u\|_\infty^p \|\psi^\tau\|_\infty^p \nonumber \\
  & \le C 2^{-j(1 - \beta^2/2)p}, \nonumber
\end{align}
where $C > 0$ depends on $p,\psi$ and $u$ but not on $j$ or $k$. Since $\E[|\alpha(\lambda)|^p]$ vanishes for all $\lambda \in \Lambda_j$ except for at most $O(2^{2j})$ of them, this concludes that $(\E[A_j^p])_{j \geq 1}$ is bounded as stated in \eqref{E:pf_Besov3}.

We now move to the proof of \eqref{E:pf_Besov4}.
Recall that $p=2N$ is an even integer and define 
\[S_j:=\sum_{\lambda\in\Lambda_j}|\nu(u \psi_\lambda)|^{2N}, \quad \quad j \geq 1.\] 
Let $\lambda_1,\lambda_2\in\Lambda_j$ and set $D_i:=\supp\psi_{\lambda_i}$, $i=1,2$. Suppose that $D_1\cap D_2=\emptyset$, set $\Delta:=\sup\{|x-y|,\,x\in D_1,y\in D_2\}$ and $\delta:=\inf\{|x-y|,\,x\in D_1,y\in D_2\}$. 
Recall the notation $\Ec(\Gamma;\xf;\yf)$ defined in \eqref{E:defEcal}. For $m=1,2$, $\E[|\nu(u \psi_{\lambda_m})|^p]$ equals 
\begin{align*}
\int_{D_m^N \times D_m^N} d\xf^m \,d\yf^m e^{\beta^2 \Ec(X;\xf^m;\yf^m)} \prod_{k=1}^N u(x^m_k) u(y^m_k) \psi_{\lambda_m}(x^m_k) \psi_{\lambda_m}(y^m_k)
\end{align*}
and $\E[|\nu(u \psi_{\lambda_1})|^p |\nu(u \psi_{\lambda_2})|^p]$ equals 
\begin{align*}
\int_{D_1^N \times D_1^N \times D_2^N \times D_2^N} & d\xf^1 \,d\yf^1 \,d\xf^2 \,d\yf^2 e^{\beta^2 \Ec(X;\xf^1 \xf^2;\yf^1 \yf^2)}
\prod_{m=1,2} \prod_{k=1}^N u(x^m_k) u(y^m_k) \psi_{\lambda_m}(x^m_k) \psi_{\lambda_m}(y^m_k).
\end{align*}
In the last display we denoted by $\xf^1 \xf^2 \in D_1^N \times D_2^N$ the concatenation of the vectors $\xf^1 \in D_1^N$ and $\xf^2 \in D_2^N$. We deduce that $\mathrm{Cov}(|\nu(u \psi_{\lambda_1})|^p ; |\nu(u \psi_{\lambda_2})|^p)$ is at most
\begin{align*}
C^N \norme{u}_\infty^{4N} \norme{\psi_\lambda}_\infty^{4N} \int_{D_1^N \times D_1^N \times D_2^N \times D_2^N} d\xf^1 \,d\yf^1 \,d\xf^2 \,d\yf^2 \Big| e^{\beta^2 \Ec(X;\xf^1 \xf^2;\yf^1 \yf^2)} - \prod_{m=1,2} e^{\beta^2 \Ec(X;\xf^m;\yf^m)} \Big|.
\end{align*}
Expanding $\Ec(\Gamma;\xf^1 \xf^2;\yf^1 \yf^2)$, we have for all $(\xf^1,\yf^1,\xf^2,\yf^2) \in D_1^N \times D_1^N \times D_2^N \times D_2^N$,
\[
\Ec(X;\xf^1 \xf^2;\yf^1 \yf^2) = \sum_{m=1,2} \Ec(X;\xf^m;\yf^m) + \sum_{1 \leq k,l \leq N} \log \frac{|x_k^1 - x_l^2| |y_k^1 - y_l^2|}{|x_k^1 - y_l^2| |y_k^1 - x_l^2|}.\]
By definition of $\Delta$ and $\delta$ and approximating the Green function by the log of the distance, the absolute value of the second sum on the right hand side is at most $2 N^2 \log \Delta/\delta$. We thus obtain that $\mathrm{Cov}(|\mu(u \psi_{\lambda_1})|^p ; |\mu(u \psi_{\lambda_2})|^p)$ is at most
\begin{align*}
    C^N \Big( \Big( \frac{\Delta}{\delta} \Big)^{2N^2 \beta^2} - 1 \Big) \norme{u}_\infty^{4N} \norme{\psi_\lambda}_\infty^{4N} \prod_{m=1,2} \int_{D_m^N \times D_m^N} d\xf^m \,d\yf^m e^{\beta^2 \Ec(X;\xf^m;\yf^m)}.
\end{align*}
By scaling, the above product of integrals is at most $C_N 2^{-(2-\beta^2/2)4jN}$. Since we also have $\norme{\psi_\lambda}_\infty^{4N} \leq 2^{4jN} \norme{\psi}_\infty^{4N}$, we have obtained that
\begin{equation}
\label{E:pf_Besov6}
    \mathrm{Cov}(|\nu(u \psi_{\lambda_1})|^p ; |\nu(u \psi_{\lambda_2})|^p) \leq C_N 2^{-(1-\beta^2/2)4jN} \Big( \Big( \frac{\Delta}{\delta} \Big)^{2N^2 \beta^2} - 1 \Big).
\end{equation}
This is the main estimate we will need in order to bound:
\begin{equation}
\label{E:pf_Besov7}
    \mathrm{Var}(S_j) = \sum_{\lambda_1, \lambda_2 \in \Lambda_j} \mathrm{Cov}(|\mu(u \psi_{\lambda_1})|^p ; |\mu(u \psi_{\lambda_2})|^p).
\end{equation}
Indeed, if $|\lambda_1-\lambda_2|\geq 2^{-j/2}$, then the associated $\Delta$ and $\delta$ satisfy $\frac{\Delta}{\delta}=1+O(2^{-j/2})$ and \eqref{E:pf_Besov6} shows that
\[
\mathrm{Cov}(|\nu(u \psi_{\lambda_1})|^p ; |\nu(u \psi_{\lambda_2})|^p) \leq C_N 2^{-(1-\beta^2/2)4jN-j/2}.
\]
The contribution of such $\lambda_1, \lambda_2$ to the sum in \eqref{E:pf_Besov7} is thus at most $C_N 2^{-j/2-4j(2-\beta^2/2)N+4j} = C_p 2^{-j/2-2j(2-2/p-\beta^2/2)p}$, recalling that $p=2N$. On the other hand, if $|\lambda_1-\lambda_2|\leq 2^{-j/2}$, then we simply use Cauchy--Schwarz inequality and \eqref{eq:alpha_scaling} to bound
$\mathrm{Cov}(|\nu(u \psi_{\lambda_1})|^p ; |\nu(u \psi_{\lambda_2})|^p) \leq C 2^{-2pj(2-\frac{\beta^2}{2})}$. Since there are at most $C2^{3j}$ couples $(\lambda_1, \lambda_2)$ such that $|\lambda_1-\lambda_2|\leq 2^{-j/2}$, this shows that the contribution of such couples to the sum in \eqref{E:pf_Besov7} is at most $C_p 2^{-j-2j(2-2/p-\beta^2/2)p}$. Altogether, we have shown that
\[
\mathrm{Var}(S_j) \leq C 2^{-j/2-2j(2-2/p-\beta^2/2)p}.
\]
Applying the Chebychev inequality, and recalling that $A_j$ is defined in \eqref{E:def_Aj}, we get that
\begin{align*}
\P\left(|A_j^{p}-\E[A_j^{p}]|\geq 2^{-j/8}\right)
&\leq 2^{j/4} \mathrm{Var}(A_j^{p})
= 2^{j/4} 2^{2j(2-2/p-\beta^2/2)p} \mathrm{Var}(S_j)
\leq C2^{-j/4}.
\end{align*}
Borel--Cantelli lemma then concludes the proof of \eqref{E:pf_Besov4}. Together with \eqref{E:pf_Besov3}, this shows that $(A_j)_{j \geq 1}$ is a.s. bounded, and thus $\mu u\in B^{-\beta^2/2}_{p,\infty}(\R^2)$ a.s., finishing the proof in this case.

\medskip

\textsc{Case 2}: $1 \leq p<\infty,q<\infty$. We want to show that $\mu \not\in B_{p,q,\loc}^{-\beta^2/2}(D)$ a.s. It is sufficient to prove that $u \nu \not\in B_{p,q}^{-\beta^2/2}(\R^2)$ a.s. for some specific choice of $u$, so let us pick $u \in C_c^\infty(D)$ which is $1$ in some square $Q \subset D$.
Arguing as in \eqref{E:pf_holder}, it is enough to show that $u\nu \not\in B_{1,q}^{-\beta^2/2}(\R^2)$ where $p$ has been replaced by 1. To conclude that $(A_j)_{j \geq 1}$ does not belong to $\ell^q$ a.s., it is then enough to show that $(\E A_j)_{j \geq 1}$ is bounded away from 0 and that $A_j - \E[A_j] \to 0$ a.s.
Let $j \geq 1$ be large enough and $\lambda = 2^{-j} k + 2^{-j-1}\tau \in \Lambda_j$ be such that the support of $\psi_\lambda$ is included in $Q$. There are at least $c 2^{2j}$ such $\lambda$. By \eqref{eq:alpha_scaling} we have
\[\E[|\alpha(\lambda)|] = 2^{j(-1+\beta^2/2)} \E\abs{\int \nu(x) \psi^\tau(x) \, dx} \ge C 2^{j(-1+\beta^2/2)},\]
for some $C > 0$ and hence $\E A_j \ge C$ for all $j$ large enough.

Thus it remains to show that $A_j - \E[A_j] \to 0$ a.s. Our basic strategy is the same as above, i.e. we will show the exponential decay of variance: $\E[|A_j - \E[A_j]|^2] \lesssim 2^{-j/2}$. We cannot however easily compute expectations such as $\E[|\alpha(\lambda_1)| |\alpha(\lambda_2)|]$ exactly since we are not looking at even powers, and therefore some extra care has to be taken in bounding the variance.

Let us begin by denoting $d(\lambda_1,\lambda_2)$ the distance between the supports of $\psi_{\lambda_1}$ and $\psi_{\lambda_2}$ and write
\[\E[|A_j - \E[A_j]|^2] = 2^{-2j(1 + \frac{\beta^2}{2})} \sum_{\lambda_1, \lambda_2 \in \Lambda_j} \E[(|\alpha(\lambda_1)| - \E[\alpha(\lambda_1)])(|\alpha(\lambda_2)| - \E[\alpha(\lambda_2)])] \eqqcolon S_1 + S_2\]
where $S_1$ and $S_2$ consist of (nearly) diagonal and off-diagonal terms:
\begin{align*}
    S_1 & = 2^{-j(2 + \beta^2)} \sum_{d(\lambda_1,\lambda_2) \le c 2^{-j}} \E[(|\alpha(\lambda_1)| - \E[\alpha(\lambda_1)])(|\alpha(\lambda_2)| - \E[\alpha(\lambda_2)])] \\
    S_2 & = 2^{-j(2 + \beta^2)} \sum_{d(\lambda_1,\lambda_2) > c 2^{-j}} \E[(|\alpha(\lambda_1)| - \E[\alpha(\lambda_1)])(|\alpha(\lambda_2)| - \E[\alpha(\lambda_2)])].
\end{align*}
Here $c > 0$ is a constant chosen so that $|x-y| \le c 2^{-j}$ whenever $x, y \in \supp \psi_{\lambda}$ for some $\lambda \in \Lambda_j$.
The diagonal terms can be handled by Cauchy--Schwarz, similarly to the Case 1 above, showing that each term is at most $C 2^{-j(2-\beta^2)}$ and there are at most $C 2^{2j}$ terms, yielding that $S_1 \lesssim 2^{-2j}$.

For the off-diagonal terms we will denote $r = d(\lambda_1,\lambda_2)$ and split $X$ as the independent sum $X = X_r + \hat{X}_r$, where $X_r$ is the $\star$-scale approximation with covariance \eqref{eq:exactly_scaling_star_scale}. In particular we have $\E[\hat{X}_r(x) \hat{X}_r(y)] = 0$ when $|x-y| \ge r$ and thus the field $\hat{X}_r$ restricted to $\supp \psi_{\lambda_1}$ is independent of itself restricted to $\supp \psi_{\lambda_2}$ and hence,
\[S_2 = 2^{-j(2 + \beta^2)} \sum_{d(\lambda_1,\lambda_2) > c2^{-j}} \E[(\E[|\alpha(\lambda_1)| | \Fc_r] - \E[|\alpha(\lambda_1)|])(\E[|\alpha(\lambda_2)| | \Fc_r] - \E[|\alpha(\lambda_2)|])],\]
where $\Fc_r$ is the $\sigma$-algebra generated by $X_r$.

Note that for $m=1,2$ we have
\[\E[|\alpha(\lambda_m)| | \Fc_r] = \E\Big[\Big| \int :e^{i\beta X_r(x)}: :e^{i\beta \hat{X}_r(x)}: u(x) \psi_{\lambda_m}(x) \, dx \Big| | \Fc_r\Big].\]
The key thing to notice is that $X_r(x)$ will typically be almost like a constant in $\supp \psi_{\lambda_m}$ if $r$ is macroscopic and $j$ is big, and hence it makes sense to write
\begin{align*}
    \E[|\alpha(\lambda_m)| | \Fc_r] & = e^{\frac{\beta^2}{2} \E[X_r(x_0)^2]} \E\Big[\Big| \int :e^{i\beta \hat{X}_r(x)}: u(x) \psi_{\lambda_m}(x) \, dx \Big|\Big] \\
    & + \E\Big[\Big| \int :e^{i\beta X_r(x)}: :e^{i\beta \hat{X}_r(x)}: u(x) \psi_{\lambda_m}(x) \, dx \Big| | \Fc_r\Big] \\
    & - \E\Big[\Big| \int :e^{i\beta X_r(x_0)}: :e^{i\beta \hat{X}_r(x)}: u(x) \psi_{\lambda_m}(x) \, dx \Big| | \Fc_r\Big]
\end{align*}
and similarly
\begin{align*}
    \E[|\alpha(\lambda_m)|] & = e^{\frac{\beta^2}{2} \E[X_r(x_0)^2]} \E\Big[\Big| \int :e^{i\beta \hat{X}_r(x)}: u(x) \psi_{\lambda_m}(x) \, dx \Big|\Big] \\
    & + \E\Big[\Big| \int :e^{i\beta X_r(x)}: :e^{i\beta \hat{X}_r(x)}: u(x) \psi_{\lambda_m}(x) \, dx \Big|\Big] \\
    & - \E\Big[\Big| \int :e^{i\beta X_r(x_0)}: :e^{i\beta \hat{X}_r(x)}: u(x) \psi_{\lambda_m}(x) \, dx \Big|\Big].
\end{align*}
By the triangle inequality we then have
\begin{align*}
    |\E[|\alpha(\lambda_m)| | \Fc_r] - \E[|\alpha(\lambda_m)|]| & \le \E\Big[\Big| \int (:e^{i\beta X_r(x)}: - :e^{i\beta X_r(x_0)}:) :e^{i\beta \hat{X}_r(x)}: u(x) \psi_{\lambda_m}(x) \, dx \Big| | \Fc_r\Big]
    \\
    & + \E\Big[\Big| \int (:e^{i\beta X_r(x)}: - :e^{i\beta X_r(x_0)}:) :e^{i\beta \hat{X}_r(x)}: u(x) \psi_{\lambda_m}(x) \, dx \Big|\Big].
\end{align*}
By Cauchy--Schwarz and Jensen's inequality we get that
\begin{align*}
    & \E[(\E[|\alpha(\lambda_1)| | \Fc_r] - \E[|\alpha(\lambda_1)|])(\E[|\alpha(\lambda_2)| | \Fc_r] - \E[|\alpha(\lambda_2)|])] \\
    & \lesssim \prod_{m=1,2} \sqrt{\E\Big[\Big| \int (:e^{i\beta X_r(x)}: - :e^{i\beta X_r(x_0)}:) :e^{i\beta \hat{X}_r(x)}: u(x) \psi_{\lambda_m}(x) \, dx \Big|^2\Big]}.
\end{align*}
We next compute the remaining second moment, getting
\begin{align*}
    & \E\Big[\Big| \int (:e^{i\beta X_r(x)}: - :e^{i\beta X_r(x_0)}:) :e^{i\beta \hat{X}_r(x)}: u(x) \psi_{\lambda_m}(x) \, dx \Big|^2\Big] \\
    & = \int (e^{\beta^2 \E[X_r(x)X_r(y)]} - e^{\beta^2 \E[X_r(x) X_r(x_0)]} - e^{\beta^2 \E[X_r(y) X_r(x_0)]} + e^{\beta^2 \E[X_r(x_0)^2]}) \cdot \\
    & \quad e^{\beta^2 \E[\hat{X}_r(x) \hat{X}_r(y)]} u(x)u(y) \psi_{\lambda_m}(x) \psi_{\lambda_m}(y) \, dx \, dy.
\end{align*}
Noting that $|x-y| \le r$ and applying \eqref{eq:exactly_scaling_star_scale} the integral is less than
\[C \int \big|e^{-2\beta^2\sqrt{\frac{|x-y|}{r}}} - e^{-2\beta^2\sqrt{\frac{|x-x_0|}{r}}} - e^{-2\beta^2\sqrt{\frac{|y-x_0|}{r}}} + 1\big| |x-y|^{-\beta^2} \psi_{\lambda_m}(x) \psi_{\lambda_m}(y) \, dx \, dy.\]
By Taylor expansion, the terms inside the first absolute value are at most $C 2^{-j/2} r^{-1/2}$, while the rest of the integral is of order $2^{-j(2 - \beta^2)}$, giving an upper bound of $r^{-\frac{1}{2}} 2^{-j(\frac{5}{2}-\beta^2)}$.
Hence we see that
\[S_2 \lesssim 2^{-j(2 + \beta^2)} \sum_{d(\lambda_1,\lambda_2) > c 2^{-j}} r^{-\frac{1}{2}} 2^{-j(\frac{5}{2} - \beta^2)} \lesssim 2^{-\frac{9}{2}j} 2^{4j} \int_{c 2^{-j}}^R \sqrt{r} \, dr \lesssim 2^{-j/2},\]
which finishes the proof of this case.

\medskip

\textsc{Case 3}: $p=+\infty$. We want to show that $\mu \not\in B_{p,q,\loc}^{-\beta^2/2}(D)$ a.s. It is sufficient to prove that $\mu u \not\in B_{p,q}^{-\beta^2/2}(\R^2)$ a.s. for some specific choice of $u$.
Consider for instance some nonnegative $u \in C^\infty_c(D)$ such that $u = 1$ on some compact set $K \subset D$ whose interior is not empty.
Let $p' \in [1,\infty]$ and $q' \in [1, \infty ]$ be the conjugates of $p$ and $q$ given by
$\frac1p + \frac1{p'} = 1$
and
$\frac1q + \frac1{q'} = 1$.
By Besov duality (see e.g. \cite[Section 2.11]{Triebel}), we have for any $f \in B_{p',q'}^{\beta^2/2}(\R^2)$,
\begin{equation}
\label{E:Besov_duality}
\norme{\mu u}_{B_{p,q}^{-\beta^2/2}(\R^2)} \geq |\mu(u f)| / \norme{f}_{B_{p',q'}^{\beta^2/2}(\R^2)}.
\end{equation}
Since we currently consider the case $p=+\infty$, we will have $p' =1$.
Let $x_0$ be any given point in the interior of $K$ and $\delta>0$ be small enough so that $Q(x_0,2\delta) \subset K$.
Let $f \in C^\infty_c(\R^2)$ be a nonnegative function which is equal to 1 in $Q(0,1)$ and which vanishes outside of $Q(0,2)$. Define $f_\delta(x) = f((x-x_0)/\delta)$, $x \in \R^2$. We apply \eqref{E:Besov_duality} to $f_\delta$. 
We first claim that
\begin{equation}
\label{E:claim_Besov}
\norme{f_\delta}_{B_{p',q'}^{\beta^2/2}(\R^2)} \leq C \delta^{2/p'-\beta^2/2}.
\end{equation}
Let us first assume that \eqref{E:claim_Besov} holds and see how we conclude from this.
Recall that $p'=1$.
Since $u=1$ on the support of $f_\delta$, the duality estimate \eqref{E:Besov_duality} applied to $f_\delta$ yields
\[
\norme{\mu u}_{B_{p,q}^{-\beta^2/2}(\R^2)} \geq c \delta^{-2+\beta^2/2} |\mu( f_\delta)|. 
\]
Using a very similar approach to what we did in the proof of Theorem \ref{th:LIL}, we can show that
\[
\limsup_{\delta \to 0}
\delta^{-2+\beta^2/2} |\mu( f_\delta)| = + \infty
\qquad \text{a.s.}
\]
The difference with the setup of Theorem \ref{th:LIL} is that we are not integrating $\mu$ against the uniform measures of concentric small squares, but against scaled versions of the nonnegative smooth test function $f$. The same line of argument can be applied in this case as well.
This shows that $\norme{\mu u}_{B_{p,q}^{-\beta^2/2}(\R^2)} = +\infty$ a.s.
We see from this argument the special role played by $p=+\infty, p' = 1$. Because of the scaling $\delta^{2/p'}$ in \eqref{E:claim_Besov}, the same reasoning would not have worked for any other value of $p$.

It remains to show \eqref{E:claim_Besov}. It is very likely that this is a standard fact, but we will derive it for the reader's convenience. Recall that we denote by $\phi$ and $\psi$ the father and mother wavelets. For $j \geq 1$ and $\lambda \in \Lambda_j$, let us denote by
\[
\alpha_\delta(\lambda) = \int \psi_\lambda f_\delta.
\]
Let $j \geq 1$. The number of $\lambda \in \Lambda_j$ such that the supports of $\psi_\lambda$ and $f_\delta$ have a nonempty intersection is at most $2^{2j} \delta^2$. Let $\lambda \in \Lambda_j$ be such. Because the integral of $\psi_\lambda$ vanishes,
\[
|\alpha_\delta(\lambda)| = \abs{ \int \psi_\lambda (f_\delta - f_\delta(x_0)) } \leq \frac{\norme{f'}_\infty}{\delta} \int |\psi_\lambda(x)| \, |x-x_0| \,dx \leq C \frac{2^{-j}}{\delta} \int |\psi_\lambda| \leq C \frac{2^{-2j}}{\delta}.
\]
If $j$ is such that $2^{-j} \geq \delta$, we simply bound $|\alpha_\delta(\lambda)| \leq \norme{f}_\infty \norme{\psi_\lambda}_{L^1} \leq C 2^{-j}$. Overall, this shows that
\[
\Big( \sum_{\lambda \in \Lambda_j} |\alpha_\delta(\lambda)|^{p'} \Big)^{1/p'} \leq C 2^{j(2/p'-1)} \delta^{2/p'} \min( 2^{-j} \delta^{-1}, 1).
\]
We conclude that 
\begin{align*}
    & \Big\|\Big\{2^{j(1 - 2/p' + \beta^2/2)}\Big(\sum_{\lambda \in \Lambda_j} |\alpha_\delta(\lambda)|^{p'}\Big)^{1/p'}\Big\}_{j \ge 1}\Big\|_{\ell^{q'}} \\
    & \leq \Big( \delta^{2q'/p'} \sum_{j =1}^{\floor{|\log_2 \delta|}} 2^{j q' \beta^2/2} + \delta^{2q'/p'-q'} \sum_{j > \floor{|\log_2 \delta|}} 2^{-j q'(1- \beta^2/2)}
    \Big)^{1/q'} \leq C \delta^{2/p' - \beta^2/2}.
\end{align*}
One can similarly control the coefficients coming from the father wavelet $\phi$ to obtain the bound \eqref{E:claim_Besov}.
This concludes the proof of Theorem \ref{thm:regularity}.
\end{proof}

\section{Convergence to white noise: proof of Theorem \ref{thm:whitenoise2}}\label{S:WN}

Throughout the section, we assume without loss of generality that $[0,1]^2 + B(0,2) \subset D$.
Recall that, as defined in \eqref{E:Wr}, we consider for all $r \in (0,1)$ the process
\[W_r(z) \coloneqq r^{\beta^2 - 5}\big(|\mu(Q(z,r))|^2 - \E|\mu(Q(z,r))|^2\big), \quad z \in [0,1]^2.\]
The goal of this section is to prove the convergence of $W_r$ as $r \to 0$ to the white noise as stated in Theorem \ref{thm:whitenoise2}.
It is in fact technically more convenient to work with an integrated version of $W_r$ that is defined, for $(x,y) \in [0,1]^2$, by
\[B_r(x,y) \coloneqq \sqrt{A} \int_{[0,x]\times[0,y]} W_r(z) \, dz\]
where $A>0$ is the normalising constant whose inverse is defined in \eqref{E:def_A_symmetrization} below.
The next result states that $(B_r)_{r \in (0,1)}$ converges to the Brownian sheet:

\begin{theorem}\label{thm:whitenoise}
  The processes $(B_r)_{r \in (0,1)}$ converge in distribution in $C^0([0,1]^2)$, but not in probability, to the Brownian sheet in $[0,1]^2$ as $r \to 0$.
\end{theorem}

Let us first see how the convergence of $(W_r)_{r \in (0,1)}$ follows:

\begin{proof}[Proof of Theorem \ref{thm:whitenoise2}, assuming Theorem \ref{thm:whitenoise}]
    By definition, for all $x,y \in [0,1]^2$, $W_r(x,y) = A^{-1/2} \partial_x \partial_y B_r(x,y)$. Because $(B_r)_{r \in (0,1)}$ converges in distribution in $C^0([0,1]^2)$, it also converges in distribution in $L^2([0,1]^2)$. By duality, it implies that $(W_r)_{r \in (0,1)}$ converges in distribution in $H^{-2}_0([0,1]^2)$ to $A^{-1/2}$ times the distributional derivative $\partial_x \partial_y$ of the Brownian sheet. The latter is a constant multiple of the white noise.
\end{proof}

The lemmas quoted in the following proof will be stated and proved in the subsequent sections (Sections \ref{SS:tightness}, \ref{SS:subsequential_limit} and \ref{SS:not_proba}).

\begin{proof}[Proof of Theorem \ref{thm:whitenoise}]
  By Lemma~\ref{lemma:tightness} we see that $(B_r)_{r \in (0,1)}$ is tight in $C([0,1]^2)$.
  To show convergence in distribution it is therefore enough to show the uniqueness of subsequential limits.
  Assume thus that $B_{r_k} \to B$ in $C([0,1]^2)$ for some sequence $r_k \to 0$. We note that by Lemma~\ref{lemma:independence} $B(s,t) - B(s',t')$ is independent of $B(s',t')$ whenever $s \ge s'$ and $t \ge t'$. Together with the variance estimate $\E B(s,t)^2 = st$ established in Lemma~\ref{lemma:variance} this is enough to characterise the Brownian sheet by \cite[Proposition~5.4]{Zakai1981}.

  If the convergence were true in probability, then in particular $B_r(1,1)$ would have to converge in probability, and indeed even in $L^2(\Omega)$ by the uniform integrability coming from the moment bound obtained in Lemma~\ref{lemma:momentbound}. This, however, is ruled out by Lemma~\ref{lemma:notcauchy}.
\end{proof}

Before stating and proving the intermediate lemmas mentioned in the previous proof, we would like to highlight a key feature of the proofs of this section.

\medskip

\textbf{Symmetrization trick.} In several places in this section, we will use ``symmetrization tricks'' that allow us to bound efficiently integrals with integrands whose signs are alternating. We give here a concrete simple example. For $x,y,u,v \in \R^2$, let
\[
\mathbf{r}(x,y,u,v) = \frac{|x-u|^{\beta^2} |y-v|^{\beta^2}}{|x-v|^{\beta^2} |y-u|^{\beta^2}}.
\]
Let us consider the following integral:
\begin{equation}\label{E:def_A_symmetrization}
    A^{-1} := \int_{\reals^2} dw \int_{Q(0,1)^2} dx \, dy \int_{Q(w,1)^2} du \, dv \frac{\mathbf{r}(x,y,u,v)-1}{|x-y|^{\beta^2} |u-v|^{\beta^2}}.
\end{equation}
 We wish to show that it is absolutely convergent as an integral over $w\in\reals^2$. The only possible source of divergence is when $|w|\to\infty$. Expanding $\mathbf{r}(x,y,u,v)-1$, we can show that for $x,y,u,v$ as in the integral above, $|\mathbf{r}(x,y,u,v)-1| =O(|w|^{-2})$ as $|w|\to\infty$. This estimate is not sufficient for the absolute integrability of the integral (we would get a logarithmic divergence), but we can exploit the symmetry of the integrand to show that it is in fact $O(|w|^{-4}$). Indeed, notice that swapping $x$ and $y$ changes $\mathbf{r}$ to $1/\mathbf{r}$: $\mathbf{r}(y,x,u,v) = 1/\mathbf{r}(x,y,u,v)$. Therefore, for all $w \in \R^2$,
\begin{align*}
f(w)
&:= \int_{Q(0,1)^2} dx \, dy \int_{Q(w,1)^2} du \, dv \frac{\mathbf{r}(x,y,u,v)-1}{|x-y|^{\beta^2} |u-v|^{\beta^2}} \\
& = \frac{1}{2} \int_{Q(0,1)^2} dx \, dy \int_{Q(w,1)^2} du \, dv \frac{1}{|x-y|^{\beta^2} |u-v|^{\beta^2}} \Big( \mathbf{r}(x,y,u,v) + \frac{1}{\mathbf{r}(x,y,u,v)} - 2 \Big).
\end{align*}
If $|w|$ is large enough and if we denote by $\mathbf{r}=\mathbf{r}(x,y,u,v)$ for $x,y,u,v$ as in the above integral, $\mathbf{r} + 1/\mathbf{r} - 2 = (\mathbf{r}-1)^2/\mathbf{r}$, so we get $f(w)=O(|w|^{-4})$ as claimed. The fact that the integral is positive also follows from the fact $(\mathbf{r}-1)^2/\mathbf{r}$ is nonnegative.


\subsection{Tightness of \texorpdfstring{$(B_r)_{r \in (0,1)}$}{B delta}}\label{SS:tightness}

In this section, we show the following tightness result:

\begin{lemma}\label{lemma:tightness}
The family $(B_r)_{r \in (0,1)}$ is tight in $C^0([0,1]^2)$.
\end{lemma}

Before proving this lemma, we first state and prove two intermediate results. The first one is a simple estimate on the Green function. The second one, on the other hand, provides a precise upper bound on the correlations of $W_r$ which is the main technical ingredient needed in the proof of Lemma \ref{lemma:tightness}. In both of these lemmas we need to bound integrals whose integrands do not have constant signs. The cancellations are important to take into account and we will do so using symmetrization tricks similar in flavour to the one below \eqref{E:def_A_symmetrization}.

\begin{lemma}\label{lemma:covariancebound}
  There exists $C > 0$ such that for any $z_1,w_1,z_2,w_2 \in [0,1]^2$ we have
  \[|G(z_1,z_2) + G(w_1,w_2) - G(z_1,w_2) - G(w_1,z_2)| \le C \frac{|z_1 - w_1| |z_2 - w_2|}{r^2},\]
  where $r = \min(|z_1 - z_2|, |z_1 - w_2|, |w_1 - z_2|, |w_1 - w_2|)$.
\end{lemma}

\begin{proof}
Let $r$ be as in the statement of lemma and denote by $\Gamma_r$ the mollification of $\Gamma$ with respect to some radially symmetric non-negative mollifier $\varphi$ whose support is included in the unit ball. Since the Green function is harmonic, we have
  \begin{align*}
    & |G(z_1,z_2) + G(w_1,w_2) - G(z_1,w_2) - G(w_1,z_2)| = |\E[(\Gamma(z_1) - \Gamma(w_1))(\Gamma(z_2) - \Gamma(w_2))]| \\
    & = |\E[(\Gamma_r(z_1) - \Gamma_r(w_1))(\Gamma_r(z_2) - \Gamma_r(w_2))]| 
    \le \sqrt{\E(\Gamma_r(z_1) - \Gamma_r(w_1))^2}\sqrt{\E(\Gamma_r(z_2) - \Gamma_r(w_2))^2}.
  \end{align*}
  By Jensen's inequality we further have
  \[|\Gamma_r(z) - \Gamma_r(w)|^2 \le |z-w|^2 \Big( \frac{1}{|z-w|} \int_0^{|z-w|} \Big| \nabla \Gamma_r\big(w + t\frac{z-w}{|z-w|}\big)\Big|^2 \, dt\Big).\]
  Lemma \ref{lemma:covariancebound} will then follow from the following bound: there exists $C>0$ such that for all $x \in [0,1]^2$ and $r \in (0,\sqrt{2})$,
  \begin{equation}
      \label{E:pf_lem_covariancebound}
      \E |\nabla \Gamma_r(x)|^2 \leq C r^{-2}.
  \end{equation}
  We now prove \eqref{E:pf_lem_covariancebound}. Let $x \in [0,1]^2$ and $r \in (0,\sqrt{2})$.
  We have
  \[
  \E |\partial_1 \Gamma_r(x)|^2
  = \int_{B(0,r)^2} G(x-y,x-z) \partial_1 \varphi_r(y) \partial_1 \varphi_r(z) dy dz.
  \]
  We want to bound the integral on the right hand side whose integrand does not have a constant sign. In order to take into account some cancellations, we will use the following \emph{symmetrization trick}.
  Since $\varphi$ is radially symmetric, $2 \E |\partial_1 \Gamma_r(x)|^2$ can be written as
  \begin{align}
  \label{E:symmetrization_trick}
      & \int_{B(0,r)^2} ( G(x-y,x-z) - G(x+y,x-z) ) \partial_1 \varphi_r(y) \partial_1 \varphi_r(z) dy dz \\
      & \leq r^{-6} \norme{\partial_1 \varphi}_\infty^2 \int_{B(0,r)^2} |G(x-y,x-z) - G(x+y,x-z)| dy dz. \nonumber
  \end{align}
  Recall that we can write $G(y,z) = -\log|y-z| + g(y,z)$ where $g$ is bounded in $[0,1]^2+B(0,\sqrt{2})$. Since
  \[
  \int_{B(0,r)^2} \abs{ \log \frac{|y+z|}{|y-z|}} dy dz = O(r^4)
  \]
  by scaling, this shows that $\E |\partial_1 \Gamma_r(x)|^2$ is bounded by $C r^{-2}$ for some $C>0$. One can obtain a similar bound for $\E |\partial_2 \Gamma_r(x)|^2$ concluding the proof of \eqref{E:pf_lem_covariancebound}. Note that without the symmetrization trick, we would have obtained an extra $|\log r|$ factor on the right hand side of \eqref{E:pf_lem_covariancebound}. This finishes the proof of the lemma.
\end{proof}

\begin{lemma}\label{lemma:momentbound}
  For any $N \ge 1$ there exists $C_N > 0$ such that for any measurable subset $U \subset [0,1]^2$ we have
  \begin{equation}\label{E:lemma_momentbound}
  \int_{U^N} |\E[W_r(z_1) \dots W_r(z_N)]| \, dz_1 \dots dz_N \le C_N |U|^{N/2}.
  \end{equation}
\end{lemma}

\begin{proof}
  We can split the integral \eqref{E:lemma_momentbound} in several parts depending on which points are further than $4r$ away from their closest neighbours. Without loss of generality it is therefore enough to check that for any $0 \le n \le N$ we have
  \[\int_{\mathcal{D}_n} |\E[W_r(z_1) \dots W_r(z_n) W_r(w_1) \dots W_r(w_{N-n})]| \lesssim |U|^{N/2},\]
  where
  \begin{align*}\mathcal{D}_n & = \{(z_1,\dots,z_n,w_1,\dots,w_{N-n}) \in U^N : |w_j - z_k| \ge 4r \text{ for all } j,k \\ & \wedge |w_j - w_k| \ge 4r \text{ for all } j \neq k \\
  & \wedge \text{for all } j \text{ we have } |z_j - z_k| \le 4r \text{ for some } k \neq j\}.
  \end{align*}
  Recalling the notation $G_D'$ defined in \eqref{E:def_GD'}, we compute
  \begin{align*}
    & |\E[W_r(z_1) \dots W_r(z_n) W_r(w_1) \dots W_r(w_{N-n})]| \\
    & = r^{N\beta^2 - 5N} \Big|\E \int \prod_{j=1}^n \big(\mu(x_j) \overline{\mu(y_j)} - e^{\beta^2 G'(x_j,y_j)}\big) \prod_{j=1}^{N-n} \big(\mu(u_j)\overline{\mu(v_j)} - e^{\beta^2 G'(u_j,v_j)}\big)\Big|,
  \end{align*}
  where the integral is over $x_j,y_j \in Q(z_j,r)$ and $u_k,v_k \in Q(w_k,r)$ with $1 \le j \le n$, $1 \le k \le N-n$.
  Recall the notation $\Ec'(\Gamma;\cdot;\cdot)$ defined in \eqref{E:defEcal'} and the formula \eqref{E:moment_iGFF}. In the display below, we denote by $\xf_\pi = (x_j)_{j \in \pi}$ and similar notations for $\xf_\tau$, $\yf_\pi, \yf_\tau$.
  Expanding the two products and computing the expected value gives
  \begin{align*}
    & r^{N\beta^2 - 5N} \Big|\int \sum_{\pi \subset \{1,\dots,n\}} \sum_{\tau \subset \{1,\dots,N-n\}} \E \Big[\prod_{j \in \pi} (\mu(x_j) \overline{\mu(y_j)}) \prod_{j \notin \pi} (-e^{\beta^2 G'(x_j,y_j)}) \\
    & \quad \cdot \prod_{j \in \tau} (\mu(u_j) \overline{\mu(v_j)}) \prod_{j \notin \tau}(-e^{\beta^2 G'(u_j,v_j)})\Big]\Big| \\
    & = r^{N\beta^2 - 5N} \Big|\int \sum_{\pi \subset \{1,\dots,n\}} e^{\beta^2 \mathcal{E}'(\mathbf{x}_\pi,\mathbf{y}_\pi)} \prod_{j \notin \pi} e^{\beta^2 G'(x_j,y_j)} \sum_{\tau \subset \{1,\dots,N-n\}} (-1)^{|\pi| + |\tau|}e^{\beta^2 \mathcal{E}'(\mathbf{u}_\tau,\mathbf{v}_\tau)} \\
    & \quad \cdot e^{\beta^2 \sum_{j \in \pi,k \in \tau} (G'(x_j,v_k) + G'(y_j,u_k) - G'(x_j,u_k) - G'(y_j,v_k))} \prod_{j \notin \tau} e^{\beta^2 G'(u_j,v_j)}\Big|.
  \end{align*}
  Note next that
  \begin{align*}
    & \mathcal{E}'(\mathbf{u}_\tau,\mathbf{v}_\tau) + \sum_{j \in \pi, k \in \tau} (G'(x_j,v_k) + G'(y_j,u_k) - G'(x_j,u_k) - G'(y_j,v_k)) + \sum_{j \notin \tau} G'(u_j,v_j) \\
    & = \sum_{j=1}^{N-n} G'(u_j,v_j) + \sum_{j,k \in \tau, j < k} \Delta_{j,k} + \sum_{j \in \pi, k \in \tau} \widetilde{\Delta}_{j,k},
  \end{align*}
  where
  \begin{align*}
    \widetilde{\Delta}_{j,k} & \coloneqq G'(x_j,v_k) + G'(y_j,u_k) - G'(x_j,u_k) - G'(y_j,v_k) \quad \text{for all } 1 \le j \le n, 1 \le k \le N - n \\
    \Delta_{j,k} & \coloneqq G'(u_j,v_k) + G'(u_k,v_j) - G'(u_j,u_k) - G'(v_j,v_k) \quad \text{for all } 1 \le j < k \le N-n.
  \end{align*}
  Thus we have
  \begin{align*}
    & |\E[W_r(z_1) \dots W_r(z_n) W_r(w_1) \dots W_r(w_{N-n})]| \\
    & = r^{N\beta^2 - 5N} \Big|\int \sum_{\pi \subset \{1,\dots,n\}} (-1)^{|\pi|}e^{\beta^2 \mathcal{E}'(\mathbf{x}_\pi,\mathbf{y}_\pi) + \beta^2 \sum_{j \notin \pi} G'(x_j,y_j)} e^{\beta^2 \sum_{j=1}^{N-n} G'(u_j,v_j)} \\
    & \quad \cdot \sum_{\tau \subset \{1,\dots,N-n\}} (-1)^{|\tau|}e^{\beta^2 \sum_{j,k \in \tau, j < k} \Delta_{j,k} + \beta^2 \sum_{j \in \pi, k \in \tau} \widetilde{\Delta}_{j,k}}\Big|.
  \end{align*}
  We now use once more a symmetrization trick by observing that the integrand is symmetric upon swapping the variables $u_k \leftrightarrow v_k$. Moreover doing the swap simply switches the sign of $\Delta_{j,k}$ and $\widetilde{\Delta}_{j,k}$. Hence we can write the above as
  \begin{align*}
    & |\E[W_r(z_1) \dots W_r(z_n) W_r(w_1) \dots W_r(w_{N-n})]| \\
    & = 2^{-(N-n)} r^{N\beta^2 - 5N} \Big|\int \sum_{\pi \subset \{1,\dots,n\}} (-1)^{|\pi|}e^{\beta^2 \mathcal{E}'(\mathbf{x}_\pi,\mathbf{y}_\pi) + \beta^2 \sum_{j \notin \pi} G'(x_j,y_j)} e^{\beta^2 \sum_{j=1}^{N-n} G'(u_j,v_j)} \\
    & \quad \cdot \sum_{\sigma \in \{0,1\}^{N-n}}\sum_{\tau \subset \{1,\dots,N-n\}} (-1)^{|\tau|}e^{\beta^2 \sum_{j,k \in \tau, j < k} (-1)^{\sigma_j + \sigma_k} \Delta_{j,k} + \beta^2 \sum_{j \in \pi, k \in \tau} (-1)^{\sigma_k} \widetilde{\Delta}_{j,k}}\Big|.
  \end{align*}
  We will next focus on estimating the two inner sums in terms of the variables $\Delta_{j,k},\widetilde{\Delta}_{j,k}$, which by Lemma~\ref{lemma:covariancebound} are bounded by $C \frac{r^2}{|w_j - w_k|^2} \le R$ (resp. by $C \frac{r^2}{|z_j - w_k|^2} \le R$) for some $R > 0$. To this end we define the analytic function
  \begin{align*}
    & f((\Delta_{j,k}), (\widetilde{\Delta}_{j,k})) \coloneqq \\
    & \sum_{\sigma \in \{0,1\}^{N-n}}\sum_{\tau \subset \{1,\dots,N-n\}} (-1)^{|\tau|}e^{\beta^2 \sum_{j,k \in \tau, j < k} (-1)^{\sigma_j + \sigma_k} \Delta_{j,k} + \beta^2 \sum_{j \in \pi, k \in \tau} (-1)^{\sigma_k} \widetilde{\Delta}_{j,k}}
  \end{align*}
  and consider its Taylor series around $0$. We have by the multinomial theorem that
  \begin{align*}
    & f((\Delta_{j,k}), (\widetilde{\Delta}_{j',k'})) \\
    & = \sum_{\ell = 0}^\infty \sum_{\sigma \in \{0,1\}^{N-n}} \sum_{\tau \subset \{1,\dots,N-n\}} (-1)^{|\tau|}\frac{\beta^{2\ell}}{\ell!} \big(\sum_{j,k} (-1)^{\sigma_j + \sigma_k} \Delta_{j,k} + \sum_{j', k'} (-1)^{\sigma_{k'}} \widetilde{\Delta}_{j',k'}\big)^\ell \\
    & = \sum_{\ell = 0}^\infty \sum_{\sigma \in \{0,1\}^{N-n}} \sum_{\tau \subset \{1,\dots,N-n\}} \sum_{\sum_{j,k} p_{j,k} + \sum_{j',k'} q_{j',k'} = \ell} \frac{(-1)^{|\tau|} \beta^{2 \ell}}{\ell!}{\ell \choose (p_{j,k})_{j,k}, (q_{j',k'})_{j',k'}} \\
    & \quad \cdot \prod_{j,k} (-1)^{p_{j,k}(\sigma_j + \sigma_k)} \Delta_{j,k}^{p_{j,k}} \prod_{j',k'} (-1)^{q_{j',k'} \sigma_{k'}}\widetilde{\Delta}_{j',k'}^{q_{j,k}},
  \end{align*}
  where the sums and products and index sets of variables are over $j,k \in \tau, j < k$ and $j' \in \pi, k' \in \tau$. We may equivalently rephrase this condition by saying that the variables $p_{j,k}$ and $q_{j',k'}$ range over the whole index set with $1 \le j < k \le N-n$ and $1 \le j' \le n$, $1 \le k' \le N-n$, that their sum is $\ell$ and with the extra condition that $p_{j,k} = 0$ if one of $j,k$ is not in $\tau$ and similarly $q_{j',k'} = 0$ if $j' \notin \pi$ or $k' \notin \tau$. Exchanging the two inner most sums then gives us
  \begin{align*}
    & \sum_{\ell = 0}^\infty \sum_{\sigma \in \{0,1\}^{N-n}} \sum_{\sum_{j,k} p_{j,k} + \sum_{j',k'} q_{j',k'} = \ell} \sum_{\substack{\tau \subset \{1,\dots,N-n\}\\ p_{j,k} \neq 0 \Rightarrow j,k \in \tau \\ q_{j',k'} \neq 0 \Rightarrow k' \in \tau}} \frac{(-1)^{|\tau|} \beta^{2 \ell}}{\ell!}{\ell \choose (p_{j,k})_{j,k}, (q_{j',k'})_{j',k'}} \\
    & \quad \cdot \prod_{j,k} (-1)^{p_{j,k}(\sigma_j + \sigma_k)} \Delta_{j,k}^{p_{j,k}} \prod_{j',k'} (-1)^{q_{j',k'}\sigma_{k'}}\widetilde{\Delta}_{j',k'}^{q_{j,k}} \\
    & = \sum_{\ell = 0}^\infty \sum_{\sigma \in \{0,1\}^{N-n}} \sum_{\substack{\sum_{j,k} p_{j,k} + \sum_{j',k'} q_{j',k'} = \ell \\ \forall 1 \le r \le N-n \text{ exists non-zero } p_{j,r}, p_{r,k} \text{ or } q_{j',r}}}\frac{(-1)^{N-n} \beta^{2 \ell}}{\ell!}{\ell \choose (p_{j,k})_{j,k}, (q_{j',k'})_{j',k'}} \\
    & \quad \cdot \prod_{j,k} (-1)^{p_{j,k}(\sigma_j + \sigma_k)} \Delta_{j,k}^{p_{j,k}} \prod_{j',k'} (-1)^{q_{j',k'}\sigma_{k'}}\widetilde{\Delta}_{j',k'}^{q_{j,k}},
  \end{align*}
  where we have used the fact that if the indices $j,k,k'$ of those $p_{j,k}$, $q_{j',k'}$ that are non-zero span a set $S \subsetneq \{1,\dots,N-n\}$, then the sum over $\tau$ can be written as a sum over the subsets of $\{1,\dots,N-n\} \setminus S$ and for any non-empty finite set $A$ we have $\sum_{\tau \subset A} (-1)^{|\tau|} = 0$.

  We note next that the sum over $\sigma$ will be $0$ unless every $\sigma_j$ appears at least twice in the two products (indeed, if $\sigma_j$ appears only once, the terms corresponding to $\sigma_j=0$ and $\sigma_j=1$ would be opposite of each other). Because the series is absolutely summable this means that we have the bound
  \[|f((\Delta_{j,k}), (\widetilde{\Delta}_{j',k'}))| \le C \sum_{p_{j,k}, q_{j',k'}} \prod_{j,k} \frac{r^{2 p_{j,k}}}{|w_j - w_k|^{2 p_{j,k}}} \prod_{j',k'} \frac{r^{2 q_{j,k}}}{|z_{j'} - w_{k'}|^{2q_{j,k}}},\]
  where $C > 0$ is some constant that depends on $N$ and $R$ and the sum is over such $p_{j,k}$ and $q_{j',k'}$ for which
  \begin{itemize}
    \item for all $1 \le r \le N-n$ we have that there exist at least two different pairs $(j,r),(r,k)$ or $(j',r)$ such that the corresponding variables are non-zero, and
    \item the variables $p_{j,k}$ and $q_{j',k'}$ are minimal in the sense that the above condition is not anymore satisfied after decreasing any of the variables $p_{j,k}$ and $q_{j',k'}$ by one.
  \end{itemize}
  One easily checks that there are only finitely many such configurations $p_{j,k}$, $q_{j',k'}$. In light of the above bound we see that it is enough to check that
  \begin{align*}
    & r^{N \beta^2 - 5N} \int_{\mathcal{D}_n} \int e^{\beta^2 \mathcal{E}'(\mathbf{x}_\pi, \mathbf{y}_\pi) + \beta^2 \sum_{j \notin \pi} G'(x_j,y_j)} e^{\beta^2 \sum_{j=1}^{N-n} G'(u_j,v_j)} \\
    & \quad \cdot \prod_{j,k} \frac{r^{2 p_{j,k}}}{|w_j - w_k|^{2 p_{j,k}}} \prod_{j',k'} \frac{r^{2 q_{j',k'}}}{|z_{j'} - w_{k'}|^{2q_{j',k'}}} \lesssim |U|^{N/2}
  \end{align*}
   where we have fixed $\pi \subset \{1,\dots,n\}$ and a configuration $p_{j,k},q_{j',k'}$ satisfying the conditions given above. As a reminder, here the first integral is over those points $z_1,\dots,z_n,w_1,\dots,w_{N-n} \in U^N$ that satisfy the nearest neighbour distance requirements given in the beginning and the second integral is again over $x_j,y_j \in Q(z_j,r)$ and $u_j,v_j \in Q(w_j,r)$. The integral over $u_j,v_j$ is $O(r^{(N-n)(4 - \beta^2)})$ and the integral over those $x_j,y_j$ for which $j \notin \pi$ is similarly $O(r^{(n - |\pi|)(4 - \beta^2)})$. The integral over those $x_j,y_j$ for which $j \in \pi$ is by a standard application of Onsager's inequality and scaling $O(r^{|\pi|(4 - \beta^2)})$. Altogether we hence have an upper boundd
  \[r^{-N} \int_{\mathcal{D}_n} \prod_{j,k} \frac{r^{2 p_{j,k}}}{|w_j - w_k|^{2 p_{j,k}}} \prod_{j',k'} \frac{r^{2 q_{j',k'}}}{|z_{j'} - w_{k'}|^{2q_{j',k'}}}.\]
  We will next bound the integrals over the $w_j$. It is helpful to view the dependencies as a graph with vertices $z_1,\dots,z_n,w_1,\dots,w_{N-n}$ such that there are $p_{j,k}$ edges between $w_j$ and $w_k$ and $q_{j',k'}$ edges between $z_{j'}$ and $w_{k'}$. The minimality condition of the configuration implies that for each of the edges one of the endpoints must have degree at most $2$. We will then proceed integrating over all of the vertices of the graph using the following operations:
  \begin{enumerate}
    \item Pick a vertex $w_j$ with degree $2$ that has a double edge to its neighbour $u$. Integrate $\frac{r^4 \1_{|w_j - u| \ge 4r}}{|w_j - u|^4}$ over $w_j$ to obtain an upper bound of order $r^2$ and remove $w_j$ from the graph.
    \item Pick a vertex $w_j$ with degree $2$ that has a double edge to its neighbour $w_k$. Integrate $\frac{r^4 \1_{|w_j - u| \ge 4r}}{|w_j - u|^4}$ over $w_j$ and use Jensen's inequality to obtain an upper bound of order $r \sqrt{|U|}$. Remove $w_j$ from the graph.
    \item Pick a vertex $w_j$ with degree $2$ that has two neighbours $u$ and $v$ and which is such that removing $w_j$ does not increase the number of connected components of the graph when restricted to the $w$-vertices. Use the AM-GM inequality to replace the term $\frac{r^4}{|w_j - u|^2 |w_j - v|^2}$ with the term $\frac{r^4}{|w_j - u|^4} + \frac{r^4}{|w_j - v|^4}$, effectively replacing the graph with two new graphs where instead of two single edges $w_j$ now has a double edge between either $u$ or $v$.
    \item Pick a vertex $w_j$ with degree $1$ and neighbour $u$. Integrate $\frac{r^2 \1_{|w_j - w_k|}}{|w_j - u|^2}$ over $u$ and use Jensen's inequality to obtain an upper bound of order $r \sqrt{|U|}$. Remove $w_j$ from the graph.
    \item Pick a vertex $w_j$ with degree $0$. Integrating over $w_j$ gives a bound of order $|U|$. Remove $w_j$ from the graph.
  \end{enumerate}
  We will now exhaust all the vertices $w_j$ in the graph by using the above operations in the following manner: Start by picking a connected component of the graph. If this connected component contains any $z$-vertices, we simply use operations (2), (3) and (4), thereby obtaining an upper bound of $(r \sqrt{|U|})^m$, where $m$ is the number of vertices $w_j$ in the connected component. If the component does not contain any $z$-vertices, we will use one operation (1) (possibly after using (3) first) and then do operations (2), (3) and (4) until only a degree $0$-vertex remains so that we finish with one operation (5). The operations (1) and (5) together give us $r^2 |U|$ and the rest of the operations give $(r \sqrt{|U|})^{m-2}$, so we again get $(r \sqrt{|U|})^m$ in total. Taking the product over all the components gives an upper bound of order $r^{N-n} |U|^{\frac{N-n}{2}}$.

  Finally we will integrate over the $z_j$ under the constraint that every $z_j$ has another $x_k$ no more than $4r$ away from it. Here one may split the integral according to the nearest neighbour function as in the proof of \cite[Lemma 3.10]{JSW}. Fixing a nearest neighbour function, recall that the nearest neighbour graph of $z_j$ consists of components which are built from two trees whose roots are joined in a 2-cycle. We then get for each such component in a straightforward manner an upper bound of order $\min(r^2,|U|)^{m-2} r^2 |U|$, where $m \ge 2$ is the number of vertices in the component. Since $\min(r^2,|U|) \le r \sqrt{|U|}$, we in fact have an upper bound of $r^m |U|^{\frac{m}{2}}$ for the given component and taking the product over all the components gives an upper bound of order $r^n |U|^{\frac{n}{2}}$, which finishes the proof.
\end{proof}

Now that we have proved Lemma \ref{lemma:momentbound}, 
Lemma \ref{lemma:tightness} will follow quickly:

\begin{proof}[Proof of Lemma \ref{lemma:tightness}]
  By Kolmogorov's tightness criterion (see e.g. \cite[Theorem 2.1.6]{MR2190038}) it is enough to check e.g. that there exists a constant $C > 0$ such that
  \[\E |B_r(z) - B_r(w)|^6 \le C |z-w|^3\]
  for all $z,w \in [0,1]^2$ and $r > 0$. This on the other hand follows from Lemma~\ref{lemma:momentbound} since
  \[\E |B_r(z) - B_r(w)|^6 \le A^3 \int_{U^6} |\E[W_r(z_1) \dots W_r(z_6)]| dz_1\dots dz_6 \lesssim |U|^3,\]
  where $A$ is defined in \eqref{E:def_A_symmetrization} and
  where $U$ is the symmetric difference between $[0,\Re(z)]\times[0,\Im(z)]$ and $[0,\Re(w)] \times [0,\Im(w)]$ and thus $|U| \lesssim |z-w|$.
\end{proof}

\subsection{Study of subsequential limits}\label{SS:subsequential_limit}

Thanks to Lemma \ref{lemma:tightness}, we now know that $(B_r)_{r \in (0,1)}$ is tight. In this section, we will study the properties of its subsequential limits. We will show in Lemma \ref{lemma:independence} the independence of their increments and compute in Lemma \ref{lemma:variance} the value of the variance.

\begin{lemma}\label{lemma:independence}
  Let $s \ge s'$ and $t \ge t'$ and let $B$ be a subsequential limit in distribution of $(B_r)_r$. Then $B(s,t) - B(s',t')$ is independent of $B(s',t')$.
\end{lemma}

\begin{proof}
  We first note that by continuity it is enough to check that for any $0 < \varepsilon < \min(s',t')$ we have $(B(s,t) - B(s',t')) \bot B(s'-\varepsilon,t'-\varepsilon)$. Let us fix such $\varepsilon$ and set $a \coloneqq s'-\varepsilon$, $b \coloneqq t'-\varepsilon$. Let us decompose our GFF as $\Gamma = \widetilde{\Gamma} + h$, where $\widetilde{\Gamma}$ is a zero-boundary GFF in $\widetilde D:= [0,s'-\frac{\varepsilon}{2}]\times[0,t'-\frac{\varepsilon}{2}]$ and $h$ is harmonic in $\widetilde D$ and independent of $\widetilde{\Gamma}$. We then define for all $r > 0$ the random variables
  \[\widetilde{B}_r \coloneqq \sqrt{A} \int_{[0,a] \times [0,b]} \tilde{h}_r(z) \, dz,\]
  where where $A$ is defined in \eqref{E:def_A_symmetrization} and
  \[\tilde{h}_r(z) \coloneqq r^{\beta^2 - 5} \Big(\big|\int_{Q(z,r)} e^{i\beta\widetilde{\Gamma}(x)} \, dx\big|^2 - \E\big|\int_{Q(z,r)} e^{i\beta\widetilde{\Gamma}(x)} \, dx\big|^2\Big).\]
 Notice that $\widetilde{B}_r$ is independent of $B_r(s,t) - B_r(s',t')$ for all $r < \varepsilon/2$. We will therefore be done if we can show that $\E |\widetilde{B}_r - B_r(a,b)|^2 \to 0$ as $r \to 0$.
 In this proof, we will write for $x,y,u,v \in \widetilde D$,
 \[
 \widetilde G'(x,y) = \E[\widetilde \Gamma(x) \widetilde \Gamma(y)] - \frac12 \log \CR(x, \widetilde D) - \frac12 \log \CR(y,\widetilde D)
 \]
 and 
 \[
 \widetilde \Ec'(x,u;y,v) = \widetilde G'(x,y) + \widetilde G'(x,v) + \widetilde G'(u,y) + \widetilde G'(u,v) - \widetilde G'(x,u) - \widetilde G'(y,v).
 \]

 We have
 \[\E |\widetilde{B}_r - B_r(a,b)|^2 = A \int_{([0,a] \times [0,b])^2} dz \, dw \E [(\tilde{h}_r(z) - W_r(z))(\tilde{h}_r(w) - W_r(w))].\]
 We will be done by dominated convergence theorem if we can show that $\E [(\tilde{h}_r(z) - W_r(z))(\tilde{h}_r(w) - W_r(w))]$ is bounded and that for $z \neq w$ it tends to $0$ as $r \to 0$.

 We will start by showing the second claim. Here it is enough to consider the four terms $\E \tilde{h}_r(z) \tilde{h}_r(w)$, $\E \tilde{h}_r(z) W_r(w)$, $\E W_r(z) \tilde{h}_r(w)$ and $\E W_r(z) W_r(w)$ separately. By looking at the proof of Lemma~\ref{lemma:momentbound} we see that $\E \tilde{h}_r(z) \tilde{h}_r(w)$ and $\E W_r(z) W_r(w)$ will for small enough $r$ have an upper bound of the form $\frac{r^2}{|z-w|^4}$. The mixed case $\E W_r(z) \tilde{h}_r(w)$ can be reduced to the $\E \tilde{h}_r(z) \tilde{h}_r(w)$ case as follows: We have
 \begin{align*}
   & \E W_r(z) \tilde{h}_r(w)
   = r^{2\beta^2 - 10} \int_{Q(z,r)^2} dx\,dy \int_{Q(w,r)^2} du\,dv \E\big[(e^{i\beta\Gamma(x)} e^{-i\beta \Gamma(y)} - e^{\beta^2 G'(x,y)})\\
   & \hspace{70pt} \cdot (e^{i\beta\widetilde{\Gamma}(u)} e^{-i\beta\widetilde{\Gamma}(v)} - e^{\beta^2 \widetilde{G}'(u,v)})\big] \\
   & = r^{2\beta^2 - 10} \int_{Q(z,r)^2} dx\,dy \int_{Q(w,r)^2} du\,dv e^{-\frac{\beta^2}{2} \E[ (h(x)-h(y))^2 ]} \E \big[(e^{i\beta\widetilde{\Gamma}(x)} e^{-i\beta \widetilde{\Gamma}(y)} - e^{\beta^2 \widetilde{G}'(x,y)})\\
   & \quad \cdot (e^{i\beta\widetilde{\Gamma}(u)} e^{-i\beta\widetilde{\Gamma}(v)} - e^{\beta^2 \widetilde{G}'(u,v)})\big] \\
   & = r^{2\beta^2 - 10} \int_{Q(z,r)^2} dx\,dy \int_{Q(w,r)^2} du\,dv e^{- \frac{\beta^2}{2} \E[(h(x)-h(y))^2]} e^{\beta^2 \widetilde{G}'(x,y) + \beta^2\widetilde{G}'(u,v)} \\
   & \quad \cdot \big(e^{\beta^2 \widetilde{G}'(x,v) + \beta^2 \widetilde{G}'(y,u) - \beta^2 \widetilde{G}'(x,u) - \beta^2 \widetilde{G}'(y,v)} - 1\big).
 \end{align*}
 Notice that the integral won't change if we swap $x$ and $y$ in the integrand. Thus the above equals
 \begin{align*}
   & \frac{1}{2} r^{2\beta^2 - 10} \int_{Q(z,r)^2} dx\,dy \int_{Q(w,r)^2} du\,dv e^{-\frac{\beta^2}{2} \E[(h(x)-h(y))^2]} e^{\beta^2 \widetilde{G}'(x,y) + \beta^2\widetilde{G}'(u,v)} \\
   & \quad \cdot \big(p(x,y,u,v) + p(x,y,u,v)^{-1} - 2\big),
 \end{align*}
 where
 \[p(x,y,u,v) = e^{\beta^2 \widetilde{G}'(x,v) + \beta^2 \widetilde{G}'(y,u) - \beta^2 \widetilde{G}'(x,u) - \beta^2 \widetilde{G}'(y,v)}.\]
 For any real number $r > 0$ we have $r + r^{-1} - 2 \ge 0$ and $\E[(h(x)-h(y))^2]$ is uniformly bounded in $[0,a] \times [0,b]$, so we see that
 \[\E W_r(z) \tilde{h}_r(w) \lesssim \E \tilde{h}_r(z) \tilde{h}_r(w).\]

 It remains to show the boundedness of $\E [(\tilde{h}_r(z) - W_r(z))(\tilde{h}_r(w) - W_r(w))]$. By Cauchy--Schwarz it is enough to show that $\E (\tilde{h}_r(z) - W_r(z))^2$ is bounded. A direct but somewhat tedious computation shows that $\E (\tilde{h}_r(z) - W_r(z))^2$ is equal to
 \begin{align*}
   r^{2\beta^2 - 10} \int_{Q(z,r)^2} dx \, dy \int_{Q(z,r)^2} du \, dv (e^{\beta^2 \widetilde{\mathcal{E}}'(x,u;y,v)} \mathfrak{A} + e^{\beta^2 \widetilde{G}'(u,v) + \beta^2 \widetilde{G}'(x,y)} \mathfrak{B})
 \end{align*}
 where
 \begin{align*}
   & \mathfrak{A} = 1 - e^{-\frac{\beta^2}{2} \E[(h(x)-h(y))^2]} - e^{\frac{\beta^2}{2} \E[(h(u)-h(v))^2]} \\
     & + e^{-\frac{\beta^2}{2} \E[(h(x)-h(u))^2] + \frac{\beta^2}{2} \E[(h(y)-h(v))2] - \frac{\beta^2}{2} \E[(h(x)-h(y))^2] - \frac{\beta^2}{2} \E[(h(u)-h(v))2] - \frac{\beta^2}{2} \E[(h(x)-h(v))^2] - \frac{\beta^2}{2} \E[(h(y)-h(u))^2]}, \\
   \mathfrak{B} & = e^{-\frac{\beta^2}{2} \E[(h(x)-h(y))^2]} + e^{-\frac{\beta^2}{2} \E[(h(u)h(v))^2]}  - e^{-\frac{\beta^2}{2} \E[(h(x)-h(y))^2] - \frac{\beta^2}{2} \E[(h(u)-h(v))^2]} - 1.
 \end{align*}

 One easily checks that both in $\mathfrak{A}$ and $\mathfrak{B}$ all the exponents are $O(r)$. Thus by expanding the exponentials as series we see that $\mathfrak{B} = O(r^2)$. For $\mathfrak{A}$ we get
 \[\beta^2 \E h(x)h(v) + \beta^2 \E h(y)h(u) -\beta^2 \E h(x)h(u) - \beta^2 \E h(y)h(v) + O(r^2),\]
 but one also easily checks that
 \begin{align*}
   & \E h(x)h(v) + \E h(y)h(u) - \E h(x)h(u) - \E h(y)h(v)  = -\E (h(x) - h(y))(h(u) - h(v)) = O(r^2),
 \end{align*}
 so in total $\E (\tilde{h}_r(z) - W_r(z))^2$ is bounded by a constant times
 \begin{align*}
   r^{2\beta^2 - 8} \int_{B(z,r)^2} dx \, dy \int_{B(w,r)^2} du \, dv (e^{\beta^2 \widetilde{\mathcal{E}'}(x,y;u,v)} + e^{\beta^2 \widetilde{G}'(x,y) + \beta^2 \widetilde{G}'(u,v)}) \lesssim 1.
 \end{align*}
\end{proof}

\begin{lemma}\label{lemma:variance}
  Let $B$ be a subsequential limit in distribution of $B_r$. Then $\E B(s,t)^2 = st$.
\end{lemma}

\begin{proof}
  Let $U = [0,s] \times [0,t]$. In this proof, we will denote by
  \[
  \mathbf{r}(x,y,u,v) = \frac{|x-u|^{\beta^2} |y-v|^{\beta^2}}{|x-v|^{\beta^2} |y-u|^{\beta^2}}
  \]
  and
  \[
  g'(x,y) = G(x,y) + \log|x-y| - \frac12 \log \CR(x,D) - \frac12 \log \CR(y,D),
  \]
  for $x,y,u,v \in D$. $g'$ is smooth and vanishes on the diagonal. We will use several times the fact that swapping $x$ and $y$ changes $r$ into $1/r$.
  A direct computation shows that
  \begin{align*}
    \E B_r(s,t)^2 & = A \int_{U^2} dz\,dw \,r^{2\beta^2 - 10} \int_{Q(z,r)^2} dx\,dy \int_{Q(w,r)^2} du\,dv \frac{e^{\beta^2 (g'(x,y) + g'(u,v))}}{|x-y|^{\beta^2} |u-v|^{\beta^2}} \\
    & \quad \cdot \big(\mathbf{r}(x,y,u,v) e^{\beta^2 (g'(x,v) + g'(y,u) - g'(x,u) - g'(y,v))} - 1\big).
  \end{align*}
  It is easy to see that in the limit as $r \to 0$ we may replace $g'(x,y) + g'(u,v)$ by $g'(z,z) + g'(w,w) = 0$ since the difference between the two is $O(r)$ and the rest of the integral is bounded by arguments similar to Lemma~\ref{lemma:momentbound}. Doing this substitution and by applying the symmetry upon swapping $x$ and $y$ (see the paragraph following \eqref{E:def_A_symmetrization} for more details) we get
  \begin{align*}
    & \frac{A}{2} \int_{U^2} dz \, dw \, r^{2\beta^2 - 10} \int_{Q(z,r)^2} dx \, dy \int_{Q(w,r)^2} du \, dv \frac{1}{|x-y|^{\beta^2}|u-v|^{\beta^2}} \frac{((1+a)\mathbf{r}(x,y,u,v) - 1)^2}{(1+a)\mathbf{r}(x,y,u,v)},
  \end{align*}
  where $a = e^{\beta^2(g'(x,v) + g'(y,u) - g'(x,u) - g'(y,v))} - 1 = O(r^2)$. Writing $\mathbf{r}=\mathbf{r}(x,y,u,v)$, we note that
  \begin{align*}
    \frac{(\mathbf{r}(1+a) - 1)^2}{\mathbf{r}(1+a)}
    & = \frac{(\mathbf{r}-1)^2}{\mathbf{r}} - \frac{(\mathbf{r}-1)^2a}{\mathbf{r}(1+a)} + \frac{2a(\mathbf{r}-1)}{1+a} + \frac{\mathbf{r} a^2}{1+a}.
  \end{align*}
  As $a = O(r^2)$ and $|\mathbf{r}-1| \lesssim \frac{r^2}{|z-w|^2}$, we see that only the first term $\frac{(\mathbf{r}-1)^2}{\mathbf{r}}$ will remain in the limit. Undoing the symmetry trick, we thus have reduced the task to finding the limit of $A \int_{U^2} dz \, dw u_r(z,w)$, where
  \begin{align*}
    & u_r(z,w) \coloneqq  r^{2\beta^2 - 10} \int_{Q(z,r)^2} dx \, dy \int_{Q(w,r)^2} du \, dv \frac{\mathbf{r}(x,y,u,v)-1}{|x-y|^{\beta^2}|u-v|^{\beta^2}}.
  \end{align*}
  We note that by the above considerations $|u_r(z,w)| = O(r)$ when $|z-w| \ge r^{1/4}$. We may therefore restrict our focus to the case where $|z-w| < r^{1/4}$. Let us consider a fixed $z$ and do the integral over $w$. By translation we may assume that $z = 0$. We have by scaling
  \begin{align}
  \notag
    & \int_{Q(0,r^{1/4})} dw r^{2\beta^2 - 10} \int_{Q(0,r)^2} dx \, dy \int_{Q(w,r)^2} du \, dv \frac{\mathbf{r}(x,y,u,v)-1}{|x-y|^{\beta^2} |u-v|^{\beta^2}} \\
    \notag
    & = \int_{Q(0,r^{-3/4})} dw \int_{Q(0,1)^2} dx \, dy \int_{Q(w,1)^2} du \, dv \frac{\mathbf{r}(x,y,u,v)-1}{|x-y|^{\beta^2} |u-v|^{\beta^2}}.
    \end{align}
As we explained in the paragraph following \eqref{E:def_A_symmetrization}, the integral w.r.t. $w$ converges absolutely as $r\to0$ due to symmetry properties. Therefore, the right hand side of the above display converges as $r \to 0$ to
    \begin{align*}
    \int_{\reals^2} dw \int_{Q(0,1)^2} dx \, dy \int_{Q(w,1)^2} du \, dv \frac{\mathbf{r}(x,y,u,v)-1}{|x-y|^{\beta^2} |u-v|^{\beta^2}}.
  \end{align*}
  This integral agrees with the inverse of the normalising constant $A$; see \eqref{E:def_A_symmetrization}. The integral over $z$ then gives $|U| = st$ as required.
\end{proof}

\subsection{The convergence does not hold in probability}\label{SS:not_proba}

\begin{lemma}\label{lemma:notcauchy}
  The sequence $(B_r(1,1))_{r \in (0,1)}$ is not Cauchy in $L^2(\Omega)$ as $r \to 0$.
\end{lemma}

\begin{proof}
  Let $0 < r < s$. We have
  \begin{align*}
    & \E |B_r(1,1) - B_s(1,1)|^2 
    = A \int_{([0,1]^2)^2} dz \, dw  \E[(W_r(z) - W_s(z))(W_r(w) - W_s(w))],
  \end{align*}
  and thus we see that it is enough to show that for fixed $s$ we have
  \[\lim_{r \to 0} \int_{([0,1]^2)^2} dz \, dw  \E[W_r(z) W_s(w)] = 0.\]
  The expectation $\E[W_r(z) W_s(w)]$ is equal to
  \begin{align*}
     r^{\beta^2 - 5} s^{\beta^2 - 5} \int_{Q(z,r)^2} dx\,dy\int_{Q(w,s)^2} du\,dv  e^{\beta^2 G'(x,y) + \beta^2 G'(u,v)}\big(e^{\beta^2 (G'(x,v) + G'(y,u) - G'(x,u) - G'(y,v))} - 1\big).
  \end{align*}
  We will consider separately the cases where $|z-w| < r^{1/4}$ and $|z-w| \ge r^{1/4}$. In the first case we note that since $G'(x,v) - G'(y,v) = O(r)$ and the same holds for $G'(y,u) - G'(x,u)$, we have
  \[|\E[W_r(z) W_s(w)]| \lesssim s^{-1}.\]
  This gets multiplied by the measure of the set $|z-w| < r^{1/4}$, which is $O(\sqrt{r})$, and we get the claim.  
  In the second case we will use the bound from Lemma~\ref{lemma:covariancebound} and the symmetry trick of swapping $x$ and $y$ to get
  \[|\E[W_r(z) W_s(w)]| \lesssim \frac{r s}{|z-w|^4},\]
  whose integral over the set $|z-w| \ge r^{1/4}$ is of the order $\sqrt{r}s$, which again goes to $0$ as $r \to 0$.
\end{proof}

\appendix

\section{A variant of Kolmogorov's continuity theorem}\label{app:kolmogorov}

The following lemma is a variant of Kolmogorov's continuity theorem. We provide a proof for the convenience of the reader.

\begin{lemma}\label{lem:Kolmogorov}
Let $\left( B(x,r), x \in D, 0 \leq r \leq 1 \right)$ be a random process indexed by a three-dimensional set. Assume that there exist $N,\eps,C>0$ such that for every $x,x' \in D, 0 \leq r,r' \leq 1$,
\[
\Expect{ \abs{B(x,r) - B(x',r')}^N } \leq C \left\{ (r \vee r') \norme{(x,r) - (x',r')} \right\}^{3+\eps}.
\]
Then there is a modification $\tilde{B}$ of $B$ satisfying: for every configuration $\omega$ and
for all $\alpha \in (0,\eps/N)$, there exists $C_\alpha(\omega) \in (0,\infty)$ such that for all $x,x' \in D, 0 \leq r,r' \leq 1$,
\[
\abs{\tilde{B}(x,r) - \tilde{B}(x',r')} \leq C_\alpha(\omega) \left\{ (r \vee r') \norme{(x,r) - (x',r')} \right\}^\alpha.
\]
\end{lemma}

\begin{proof}
In this proof we will denote by $\Dc_2^n$ (resp. $\Dc_1^n$) the set of two-dimensional (resp. one-dimensional) dyadic points of generation $n$ lying in $D$ (resp. $[0,1]$). We will write $\Dc_2 = \cup_{n \geq 1} \Dc_2^n$ and $\Dc_1 = \cup_{n \geq 1} \Dc_1^n$. Take $\alpha < \eps/N$. If $x \in \Dc_2^n, r \in \Dc_1^n, r' = r+2^{-n}$, we have
\[
\Prob{ \abs{ B(x,r) - B(x,r') } \geq (r \vee r')^\alpha 2^{-n\alpha} }
\leq
C \left( (r \vee r') 2^{-n} \right)^{3+\eps-\alpha N}
\leq
C 2^{-n (3+\eps-\alpha N) }.
\]
By a union bound and because $\eps - \alpha N >0$, we get that
\[
\sum_{n \geq 1} \Prob{ \exists x \in \Dc_2^n: r \in \Dc_1^n, \abs{ B(x,r) - B(x,r+2^{-n}) } \geq (r+ 2^{-n})^\alpha 2^{-n\alpha} }
\]
is finite. By the Borel-Cantelli lemma, it implies that
\[
\sup \left\{ (r+2^{-n})^{-\alpha} 2^{n\alpha} \abs{ B(x,r) - B(x,r+2^{-n}) }: n \geq 1, x \in \Dc_2^n, r \in \Dc_1^n \right\}
\]
is almost surely finite (the quantity above is clearly finite for a finite number of $n$).
We can proceed in the same way for the first component to show that
\[
\sup \left\{
\frac{\abs{ B(x,r) - B(x',r') }}{\left\{ (r \vee r') \norme{(x,r) - (x',r')} \right\}^\alpha} :
\begin{array}{l}
n \geq 1, x,x' \in \Dc_2^n, \\
r,r' \in \Dc_1^n, \\
\norme{(x,r)-(x',r')}_{L^1} = 2^{-n}
\end{array} \right\}
\]
is almost surely finite. We denote this quantity by $K_\alpha$. Now consider $(x,r),(x,r') \in \Dc_2 \times \Dc_1$. Assume for instance that $r < r'$. Consider $p \geq 1$ s.t. $2^{-p} \leq r' - r \leq 2^{-p+1}$ and $k$ the smallest integer such that $r \leq k 2^{-p}$. Let $\ell \geq 0$ be s.t. $(x,r),(x,r') \in \Dc_2^{p+\ell} \times \Dc_1^{p+\ell}$. We can write $r$ and $r'$ as
\[
r = k2^{-p} - \eps_12^{-p-1} - \dots - \eps_\ell 2^{-p-\ell}
\mathrm{~and~}
r' = k2^{-p} + \eps'_02^{-k} + \eps'_12^{-p-1} + \dots + \eps'_\ell 2^{-p-\ell}
\]
with $\eps_i, \eps'_i \in \{0,1\}$. Define for all $i=0, \dots, \ell$,
\[
r_i = k2^{-p} - \eps_12^{-p-1} - \dots - \eps_i 2^{-p-i}
\mathrm{~and~}
r'_i = k2^{-p} + \eps'_02^{-k} + \eps'_12^{-p-1} + \dots + \eps'_i 2^{-p-i}.
\]
In particular, $r_i,r'_i \in \Dc_1^{p+i}$ and $r_i - r_{i+1}, r'_{i+1} - r'_i \in \{0,2^{-p-i-1}\}$. Notice that for all $i = 0 \dots \ell-1, r_i \vee r'_i \leq r'$. Similarly, there exists a sequence $(x_i, i=0 \dots 2\ell)$ with $x_{2\ell} = x$ such that
\begin{align*}
\forall i=0, \dots, \ell-1, x_{2i},x_{2i+1} \in \Dc_2^{p+i}, \norme{x_{2i} - x_{2i+1}}_{L^1} \in \{0, 2^{-p-i}\}.
\end{align*}

We can bound from above $\abs{B(x,r) - B(x_0,r_0)}$ by
\begin{align*}
& \sum_{i=0}^{\ell-1} \abs{B(x_{2i+2},r_{i+1}) - B(x_{2i+1},r_{i+1})}
+ \abs{B(x_{2i+1},r_{i+1}) - B(x_{2i},r_{i+1})}  + \abs{B(x_{2i},r_{i+1}) - B(x_{2i},r_{i})} \\
& \leq 3 K_\alpha (r')^\alpha \sum_{i=0}^{\ell-1} 2^{-(p+i+1)\alpha}
\leq 3 \frac{2^{-\alpha}}{1-2^{-\alpha}} K_\alpha (r')^\alpha 2^{-p \alpha}.
\end{align*}
The same bound holds true for $\abs{B(x,r') - B(x_0,r'_0)}$ and
\begin{align*}
\abs{B(x,r) - B(x,r')} & \leq 6 \frac{2^{-\alpha}}{1-2^{-\alpha}} K_\alpha (r')^\alpha 2^{-p \alpha} + \abs{ B(x_0,r_0,s_0) - B(x_0,r'_0,s_0)} \\
& \leq \left( 1 + 6\frac{2^{-\alpha}}{1-2^{-\alpha}} \right) K_\alpha (r')^\alpha 2^{-p \alpha}
\leq \left( 1 +6 \frac{2^{-\alpha}}{1-2^{-\alpha}} \right) K_\alpha (r')^\alpha (r'-r)^\alpha.
\end{align*}
A similar treatment can be applied when $x$ varies as well. If we define for all $(x,r) \in \Dc_2 \times \Dc_1$,
\[
\tilde{B}(x,r) =
\left\{
\begin{array}{ll}
B(x,r) & \mathrm{if~} K_\alpha < \infty \\
0 & \mathrm{otherwise}
\end{array}
\right.,
\]
and for all $(x,r) \in D \times [0,1]$,
\[
\tilde{B}(x,r) = \lim_{\substack{(x',r') \to (x,r)\\(x',r') \in \Dc_2 \times \Dc_1}} \tilde{B}(x',r'),
\]
then $\tilde{B}$ satisfies the requirements of the lemma.
\end{proof}

\section{A tail estimate for $|\mu^*(Q(0,1))|$}\label{app:tail}

In this appendix, we give some details on the estimate \eqref{eq:rem tail} from \cite{Leble2017}. Recall that we denote by $\mu^*$ the imaginary chaos associated to the whole-plane GFF $\Gamma^*$. The main result of \cite{Leble2017} is the existence of a constant $c^{**}(\beta) \in \R$ such that
\begin{equation}
\label{eq:app_LSZ}
\E \abs {\mu^*([0,1]^2) }^{2N} = N^{\beta^2 N/2} e^{c^{**}(\beta)N + o(N)}, \qquad \text{as} \quad N \to \infty.
\end{equation}
See \cite[Corollary 1.3]{Leble2017}.
Let us define the constant $c^*(\beta) >0$ by
\begin{equation}
\label{E:app_c*}
c^*(\beta) = \frac{\beta^2}{2} e^{-1-2c^{**}(\beta)/\beta^2} \qquad \text{so that} \qquad
c^{**}(\beta) = - \frac{\beta^2}{2}\Big(1+\log \frac{2c^*(\beta)}{\beta^2}\Big).
\end{equation}
In \cite[Corollary A.2]{Leble2017}, the authors argue that \eqref{eq:app_LSZ} yields the following tail estimate:
\begin{equation}\label{E:app_tail_lim}
\lim_{t \to \infty} t^{-4/\beta^2} \log \Prob{ |\mu^*(Q(0,1))| > t} = -c^*(\beta).
\end{equation}
The upper bound follows directly from \eqref{eq:app_LSZ} by optimising some Markov's inequality. The lower bound however necessitates more care and the proof in \cite{Leble2017} actually seems to contain a small gap. 
Indeed, two random variables $X$ and $Y$ are considered (with the notations therein $X = Y_\delta -C_2$ and $Y = |\int_D e^{i \beta h(x)} dx|^2$). These variables satisfy for all $k \geq 1$, $\E[X^k] \leq \E[Y^k]$. It is then inferred that $X$ is stochastically dominated by $Y$ in the sense that $\P(X > t) \leq \P(Y>t)$ for all $t$. This implication is however erroneous in general as one can see with the example $X = 1/3$ a.s. and $Y$ a Bernoulli random variable with $\P(Y=1) = 1-\P(Y=0) = 2/3$.
Since several main results of the current paper crucially use this tail estimate, we give an alternative elementary way of obtaining \eqref{E:app_tail_lim} from \eqref{eq:app_LSZ}.

Let us also mention that we have corrected a typo in \cite{Leble2017} where \eqref{E:app_tail_lim} is stated with $t^{2/\beta^2}$ rather than $t^{4/\beta^2}$. Indeed, with their notations, the $2k$-th moment of the absolute value of the imaginary chaos is equal to $Z_{k,\beta^2}$ rather than $Z_{k,2\beta^2}$, contrary to what is stated in \cite[Lemma A.1]{Leble2017}. Now, at the critical value $\beta = \sqrt{2}$, the tail \eqref{E:app_tail_lim} has a Gaussian behaviour. This is consistent with the fact that the imaginary chaos becomes a white noise at $\beta = \sqrt{2}$ (see \cite{complexGMC,JSW}).

\medskip

We now explain our proof of the lower bound of \eqref{E:app_tail_lim}.
Let $\nu$ be the law of $\abs{ \mu^*(Q(0,1))}$, $\alpha = 4/\beta^2$ and $c^* = c^*(\beta)$ be the constant defined in \eqref{E:app_c*}.
As mentioned, by a direct application of Markov's inequality, 
for all $d^* < c^*$, if $t$ is large enough
$\nu [t, \infty) \leq e^{-d^* t^\alpha}$.
We want to show the reverse inequality, i.e. that
\begin{equation}
    \label{E:app_tail_lower}
    \liminf_{t \to \infty} \frac{\log \nu [t,\infty)}{t^\alpha} \geq - c^*.
\end{equation}
For all $\eps>0$ small enough and $d^* < c^*$, letting $k(t)$ be the unique even integer in $[d^* \alpha t^\alpha, d^* \alpha t^\alpha + 2)$, we first claim that
\begin{equation}
\label{eq:appendix_tail}
\abs{ \int_0^\infty \frac{x^{k(t)}}{t^{k(t)}} \nu(dx) - \int_{(1-\eps)^{1/\alpha} t}^{(1+\eps)^{1/\alpha} t} \frac{x^{k(t)}}{t^{k(t)}} \nu(dx) }
\leq e^{-(1+o(1)) d^* (1+\eps^2/4)t^\alpha},
\end{equation}
where $o(1) \to 0$ as $t \to \infty$.

\begin{proof}[Proof of \eqref{eq:appendix_tail}]
Since $-\log(1+u) + 1 + u = 1 + u^2/2 + O(u^3)$ as $u \to 0$, there exists $u_0>0$ s.t. for all $u \in (-u_0,u_0)$, $-\log(1+u) + 1 + u \geq 1+ u^2/4$. We now fix $\eps \in (0,u_0)$ and $d^* < c^*$. There exists $t_0>0$ large enough such that for all $t \geq t_0$, $\nu [t, \infty) \leq \int_t^\infty \frac{1}{Z} e^{-d^* x^\alpha} dx$ where $Z$ is the renormalising constant $Z = \int_0^\infty e^{-d^* x^\alpha} dx$. We thus have the stochastic domination
\begin{equation}
\label{E:app_domination}
\nu(dx) \prec \frac{1}{Z} e^{-d^* x^\alpha} dx, \qquad \text{for} \quad x\geq t_0.
\end{equation}
We decompose the left hand side of \eqref{eq:appendix_tail} into several parts depending on whether $x \le t_0$, $x \in [t_0,(1-\eps)^{1/\alpha} t]$ or $x \ge (1+\eps)^{1/\alpha} t$. For the first part,
there exists $c >0$ depending on $\alpha, d^*, t_0$ such that
\[
\int_0^{t_0} \frac{x^{k(t)}}{t^{k(t)}} \nu(dx) \leq t_0^{k(t)} t^{-k(t)} \leq e^{-c t^\alpha \log t}.
\]
Noticing that $x \mapsto x^{k(t)} e^{-d^* x^\alpha}$ is increasing on $[ 0, ( \frac{k(t)}{d^* \alpha})^{1/\alpha} ]$ and then decreasing, we also have:
\begin{align*}
\int_{t_0}^{(1-\eps)^{1/\alpha} t} \frac{x^{k(t)}}{t^{k(t)}} \nu(dx)
& \leq \frac{1}{Z}
\int_{t_0}^{(1-\eps)^{1/\alpha} t} \frac{x^{k(t)}}{t^{k(t)}} e^{-d^* x^\alpha} dx
\leq \frac{1}{Z} (1-\eps)^{1/\alpha} t (1-\eps)^{k(t)/\alpha} e^{-d^* (1-\eps) t^\alpha} \\
& = e^{-(1+o(1)) \left( - d^* \log (1-\eps) + d^*(1-\eps) \right) t^\alpha}
\leq e^{-(1+o(1)) d^* (1+\eps^2/4) t^\alpha}
\end{align*}
because $\eps < u_0$.
Similarly (essentially change $\eps \leftarrow -\eps$ in the computations above)
\[
\int_{(1+\eps)^{1/\alpha} t}^{t^2} \frac{x^{k(t)}}{t^{k(t)}} \nu(dx)
\leq e^{-(1+o(1)) d^* (1+\eps^2/4)t^\alpha}.
\]
The integral from $t^2$ to $\infty$ is even smaller and putting things together, we have obtained \eqref{eq:appendix_tail}.
\end{proof}

Let $\eps >0$ be small and $d^* < c^*$ close enough to $c^*$ so that $\eps^2/4 > \log ( c^*/d^*).$
By \eqref{eq:app_LSZ},
\begin{align*}
\int_0^\infty \frac{x^{k(t)}}{t^{k(t)}} \nu(dx) & = t^{-k(t)} \exp \left( \frac{1}{\alpha} k(t) \log \frac{k(t)}{2} - \frac{1}{\alpha} \left( 1 + \log \frac{\alpha c^*}{2} \right) k(t) + o(k) \right) \\
& = \exp \left( - (1+o(1)) d^* \left( 1 + \log \frac{c^*}{d^*} \right) t^\alpha \right).
\end{align*}
Because we have chosen $\eps$ and $d^*$ so that $\eps^2/4 > \log (c^*/d^*)$, combining the above estimate with \eqref{eq:appendix_tail} leads to
\[
\int_{(1-\eps)^{1/\alpha} t}^{(1+\eps)^{1/\alpha} t} \frac{x^{k(t)}}{t^{k(t)}} \nu(dx) \geq \exp \left( - (1+o(1)) d^* \left( 1 + \log \frac{c^*}{d^*} \right) t^\alpha \right).
\]
In particular,
\begin{align*}
\nu \left[ (1-\eps)^{1/\alpha} t , \infty \right)
& \geq (1+\eps)^{-k(t)/\alpha}
\int_{(1-\eps)^{1/\alpha} t}^{(1+\eps)^{1/\alpha} t} \frac{x^{k(t)}}{t^{k(t)}} \nu(dx) \\
& \geq \exp \left( - (1+o(1)) d^* \left( 1 + \log \frac{c^*}{d^*(1-\eps)} \right) t^\alpha \right).
\end{align*}
We have thus shown that
\[
\liminf_{t \to \infty} \frac{\log \nu [t,\infty)}{t^\alpha} \geq - \frac{d^*}{1-\eps} \left( 1 + \log \frac{c^*}{d^*(1-\eps)} \right).
\]
By letting $d^* \to c^*$ and then $\eps \to 0$, we obtain \eqref{E:app_tail_lower} as wanted.

\bibliographystyle{alpha}
\bibliography{bibliography}

\end{document}